\documentclass[fleqn]{article}
\usepackage[top = 2cm, left = 2cm, right = 2cm, bottom = 2cm]{geometry}
\usepackage[utf8]{inputenc} 
\usepackage[hidelinks]{hyperref}
\usepackage[english]{babel}
\usepackage{graphics}
\usepackage{color}
\usepackage{multicol}
\usepackage{amsfonts}
\usepackage{amsmath,amsthm,amssymb,mathtools,tikz,hyperref}
\usepackage[nottoc]{tocbibind}
\usepackage{float}
\usepackage{amsfonts}
\usepackage{bm}
\usepackage{subcaption}
\usepackage{siunitx}
\usepackage{bbold}
\usepackage{csquotes}
\usepackage{enumerate}
\usepackage[round]{natbib}
\usepackage[appendix=inline,bibliography=common]{apxproof}

\renewcommand{\le}{\leqslant}
\renewcommand{\leq}{\leqslant}
\renewcommand{\ge}{\geqslant}
\renewcommand{\geq}{\geqslant}

\newcommand{\sign}{\text{sgn}}

\newcommand{\E}{\mathsf{E}\,}

\newcommand{\var}{\mathsf{Var}\,}
\newcommand{\cov}{\mathsf{Cov}\,}

\newcommand{\indi}{\mathbb{1}}
\newcommand{\Z}{\mathbb{Z}}
\newcommand{\N}{\mathbb{N}}
\newcommand{\R}{\mathbb{R}}

\theoremstyle{plain}
\newtheoremrep{thm}{Theorem}[section]
\newtheoremrep{lemma}[thm]{Lemma}

\newtheorem{ass}{Assumption}
\newtheorem{example}{Example}
\newtheorem{rem}{Remark}

\newtheorem{prop}{Proposition}[section]

\newcommand{\dto}{\overset{\mathcal{D}}{\longrightarrow}}

\title{Fully Functional Weighted Testing for Abrupt and Gradual Location Changes in Functional Time Series}
\author{Claudia Kirch, Hedvika Ranošová,  Martin Wendler}
\date{\today}

\begin{document}
\maketitle
\begin{abstract}
Change point tests for abrupt changes in the mean of functional data, i.e., random elements in infinite-dimensional Hilbert spaces, are either based on dimension reduction techniques, e.g., based on principal components, or directly based on a functional CUSUM (cumulative sum) statistic. The former have often been criticized as not being fully functional and losing too much information. On the other hand, unlike the latter, they take the covariance structure of the data into account by weighting the CUSUM statistics obtained after dimension reduction with the inverse covariance matrix.  In this paper, as a middle ground between these two approaches, we propose an alternative statistic that includes the covariance structure with an offset parameter to produce a scale-invariant test procedure and to increase power when the change is not aligned with the first components. We obtain the asymptotic distribution under the null hypothesis for this new test statistic, allowing for time dependence of the data. Furthermore, we introduce versions of all three test statistics for gradual change situations, which have not been previously considered for functional data, and derive their limit distribution. Further results shed light on the asymptotic power behavior for all test statistics under various ground truths for the alternatives. \\

\end{abstract}

\section{Introduction} \label{sec_introduction}
The statistical analysis of functional data, both parametrically and non-parametrically, has received considerable attention in the last decades, see \cite{ferraty2006nonparametric,horvath2012inference,kokoszka2017introduction} for some recent textbooks or \cite{gertheisssurvey24} for a recent survey article. In particular, functional data analysis provides a rigorous mathematical framework to analyze many relevant high-dimensional data sets, taking the complex dependencies both across time and across location  into account, that is often present in such data sets. This includes, e.g., the analysis of neuroscience (fMRI-)data sets \citep{aston2012evaluating,stoehr2021detecting}, yearly temperature curves \citep{berkes2009detecting}, ocean drifter data \citep{gnettner2024symmetrisation} or financial data sets \citep{KOKOSZKA2024104426}, to name but a few.

Change point analysis deals with the question, whether data observed over a time course remains stationary or whether its stochastic structure changes over time. If changes are not taken into account in the subsequent analysis, then this statistical analysis is often no longer reliable -- a phenomenon related to what is known as concept drift in the AI community and known to  lead to a degeneration of predictive power in AI systems \citep{ADAMS2023102175,hinder2024one}. Change point analysis has a long tradition in statistics dating back to the 1950s \citep{Page1}, but has received considerable attention in the last decades, see, e.g.\, \cite{cho2024data,aue2024state,horvath2024change}  for recent overviews.

Not surprisingly, there are also several works dealing with the detection of changes in functional data, e.g.\ \cite{harrisetal2022,NEURIPS2022_f0add74c,horvathetal26functional,bastian2025gradual,paparoditisetal26} to name a few recent papers. The text books \cite{horvath2012inference,horvath2024change} each contain a chapter on change point detection for functional data.

To set the stage, consider the following  \textit{signal plus noise} model 
\begin{align}
    \label{eq_model}
X_i(z) = \mu_i(z) + \varepsilon_i(z), \quad z \in \mathcal{Z}, \ i=1,\ldots,n,
\end{align}
where $\{\varepsilon_i(\cdot)\}$ is a centered non-degenerate stationary sequence with values in the space of $\mathcal{H}=L^2(\mathcal{Z})$ of square integrable function for a domain $\mathcal{Z}$ equipped with a $\sigma$-finite measure $\nu$. $\mathcal{H}=L^2(\mathcal{Z})$ is real separable Hilbert space with an inner product $\langle f,g\rangle=\int_{\mathcal{Z}}fg\,d\nu$ that induces the norm $\|f\|=\sqrt{\langle f,f\rangle}$. Additionally, we require $\E \|\varepsilon_1\|^{{k}} < \infty$ for some $k>2$.

We aim to test the null hypothesis of the stationarity of the mean
$H_0: \mu_1 = \ldots = \mu_n$
against a general alternative 
 \begin{equation}
         H_A: \mu_i = \mu + g \left( \frac in \right) \cdot \Delta \label{alt_grad}
      \end{equation}  
where $g:[0,1] \to \R$ is Riemann integrable, non-constant, and $g(0)=0$. All parameters $\mu, \Delta \in L^2(\mathcal{Z})$ and $g$ are taken as unknown, and the assumptions on $\varepsilon_i$ are summarized in Assumptions~\ref{ass_covop} and \ref{ass_limit_FF} below. 
Classical alternatives are given by the (abrupt) AMOC (at-most-one-change) alternative, e.g.\ considered in \cite{berkes2009detecting},  where $\mu_i$ can be rewritten as
     \begin{equation}
        \mu_i = \mu+  \indi [i >  \theta n  ] \cdot \Delta, \label{alt_amoc}
     \end{equation}
  for $0< \theta\le 1$,    or the \emph{epidemic} alternative, for example considered in \cite{aston2012detecting}, with $0<\theta_1\le \theta_2\le 1$,
      \begin{equation}
          \mu_i = \mu+  \indi [  \theta_1 n  < i \leq   \theta_2 n  ] \cdot \Delta. \label{alt_ep}
      \end{equation}
      In this paper, we will consider abrupt changes in addition to \emph{(delayed) gradual change alternatives} defined by 
      \begin{equation}
          \mu_i = \mu+ h\left(\frac{i }{n}-\theta\right) \cdot\Delta\label{alt_d_g}
      \end{equation}
      for some smooth function $h:\mathbb{R}\to\mathbb{R}$ with $h(x)=0$ for $x\le 0$, $0<\theta\le 1$ and (w.l.o.g.) $\inf \{x: h(x)\neq 0\}=0$.
 In all of the above examples $\theta$ respectively $\theta_j$, $j=1,2$, are called \emph{change points}.
      
      A different type of gradual change alternative for functional data was recently considered by \cite{bastian2025gradual} using the supremum norm.

\paragraph*{Contributions and Outline:}

In Section~\ref{ssec_change-point_model}, we propose a new test statistic for the AMOC situation that takes the underlying covariance structure of the data into account while still being fully functional. We motivate the statistic as a compromise between a previously proposed approach based on dimension reduction and previously proposed fully functional statistics. In Section~\ref{ssec_null_asymptotics} we derive the corresponding asymptotic null distribution, complemented by asymptotic power discussions for various alternatives beyond the correctly specified AMOC situation in Section~\ref{ssec_power_considerations}. In Section~\ref{sec_gradual_change}, we propose a broader class of statistics designed for gradual changes for each of the three types of test statistics -- dimension reduction, classical fully functional and the new weighted fully functional approach. Again, we derive the null asymptotics of the proposed statistics and consider their asymptotic power in general cases including misspecification of the type of change.  We derive all of our theoretical results not only for time-independent functional data but allow for time-dependence.

\section{Fully Functional Weighted Tests Designed for Abrupt Changes} \label{sec_abrupt_change}
All test statistics discussed in this section are constructed based on the AMOC alternative as detailed in \eqref{alt_amoc}. We first motivate a new test statistic that takes the covariance structure of the data into account, while still being fully functional. Then, we analyze its behavior both under the null hypothesis and general alternatives not restricted to the AMOC alternative.
Extensions to a model of the \textit{delayed gradual change} are presented in Section~\ref{sec_gradual_change}.

\subsection{Test Statistics} \label{ssec_change-point_model}
There are essentially two competing approaches, one by dimension reduction, the other fully functional, that have previously been proposed for testing for a change in the mean of functional data, which are reconciled by a new proposal consisting of a weighted functional approach.
\subsubsection{Dimension Reduction Approach}
\cite{berkes2009detecting} propose to use a dimension reduction technique to transform the functional data into a multivariate data set and then subsequently apply classical change point methods for multivariate data.
To elaborate, let $d \in \N$ and choose $d$ orthogonal functions $\{w_p \}_{p=1}^d$. The random functions $X_1,\ldots, X_n$ are then transformed into $d$-dimensional random vector $\boldsymbol{\eta}_1,\ldots,\boldsymbol{\eta}_n$ with $\boldsymbol{\eta}_i=(\langle X_i, w_1\rangle,\ldots, \langle X_i, w_p \rangle)'$. Due to the linear nature of projections, the transformed data also has a change in expectation of the same type, e.g.\ AMOC or gradual change, with the same change point. This allows us to use the CUSUM test statistic for multivariate data based on the centered partial sums $\mathbf{S}_n$ as
    \begin{equation}\label{eq_PC_quadratic_form}
    T^2_{n,PC}=T^2_{n,\operatorname{PC}}(w_1,\ldots,w_d)=    \max_{1 \leq k < n} \frac{1}{n}\, \mathbf{S}_n\left(\frac{k}{n}\right)' \Sigma^{-1} \,\mathbf{S}_n\left(\frac{k}{n}\right),\quad
        \mathbf{S}_n(t)=\sum_{j=1}^{\lfloor tn \rfloor} \left(\boldsymbol{\eta}_j - \boldsymbol{\bar{\eta}}_n\right) 
    \end{equation}
    where $\boldsymbol{\bar{\eta}}_n = \frac{1}{n} \sum_{i=1}^n \boldsymbol{\eta}_i$. The purpose of weighting the partial sum process with the inverse of an estimator of the long-run covariance $\Sigma^{-1}$ of $\{\boldsymbol{\eta}_i \}_{i \in \Z}$ is two-fold: First, it balances the signal-to-noise ratio in an ideal way. Secondly, it leads to a distribution-free limit under the null hypothesis.     Here, the change is detected if it is not orthogonal to the subspace generated by $\{w_p \}_{p=1}^d$.

     Following classical dimension reduction techniques, \cite{berkes2009detecting} propose to use functional principal components (fPCA) as a data-dependent choice for dimension reduction, that is, they chose the subspace that explains the most variation among all $d$-dimensional subspaces. Clearly, other dimension reduction techniques are also possible (see \cite{aston2012detecting}).   
   The principal components  are then estimated as the eigenfunctions of the sample covariance function 
    $$\hat{c}(t,s) = \frac{1}{n}\sum_{i=1}^n (X_i(t) - \bar{X}_n(t))(X_i(s) - \bar{X}_n(s)).$$
    In case of i.i.d.\ errors, the matrix $\Sigma$ corresponding to the PCA approach is diagonal, and its estimator includes the estimated first $d$ eigenvalues of $C_0$, which is the covariance operator of the noise term $\varepsilon_1(\cdot)$. In the presence of dependent errors, the components are no longer uncorrelated due to temporal dependence, and the diagonal matrix must be replaced by an estimator of the long-run covariance $\Sigma = \sum_{r=-\infty}^\infty \E \boldsymbol{\eta}_0\boldsymbol{\eta}_r'$ (\cite{aston2012detecting}). The diagonal structure of the variance is kept if we replace the short-run covariance operator with the long-run covariance operator in the dimension reduction procedure (\cite{torgovitski2015darling}). Additionally, \cite{torgovitski2015detecting} recalculates the first component so it is better aligned with the direction of the change.
 Under alternatives the estimated PCA-basis will no longer approximate the PCA-basis of the underlying errors, but will be contaminated by the presence of the change. However, the power of the procedure will benefit from this contamination in the following sense: Any change that is not orthogonal to the principal curves of the error process $\{\varepsilon_i(\cdot)\}$ (Theorem 4.1 (a) in \cite{aston2012detecting}) is detectable. Furthermore, if the change is large enough, the principal components procedure will detect the change even if it does not align with the uncontaminated principle components, as the subspace generated by the change becomes more prominent than the first principal curves (Theorem 4.1 (b) in \cite{aston2012detecting}).  
 
\subsubsection{Fully Functional Version} 
The previous approach has often been criticized as not being fully functional due to the dimension reduction involved. Furthermore, while large enough changes are guaranteed to be detectable, this is not true for small (fixed) changes that are orthogonal to the non-contaminated principal curves.
This motivates the following alternative approach, which is based on the CUSUM test statistic in the functional domain, disregarding the covariance operator, with the   test statistic  defined as 
    $$T_{n, \operatorname{FF}}=\max_{1\leq k < n} \frac{1}{\sqrt{n}}\left\|\sum_{i=1}^k X_i - \frac{k}{n} \sum_{i=1}^n X_i\right\|. $$ 
  Because the functional covariance operator is not invertible, this test statistic cannot be weighted by the inverse of the operator in the same way as the multivariate statistic is. In consequence, the limit distribution is no longer pivotal but depends on the associated covariance operator. In fact, 
starting with the $\mathcal{H}$-valued version of the functional central limit theorem for the error terms $\varepsilon_i$ towards an 
     $\mathcal{H}$-valued Brownian motion $\{W_{C_{\varepsilon}}(t),t \in [0,1]\}$ with an associated covariance operator $C_{\varepsilon}$ as explained in \cite{sharipov2016sequential}, the limit is easily derived; see Theorem~\ref{thm_fclt} in the appendix.
   
\subsubsection{Weighted functional approach}
The behavior of both approaches depends on the change function $\Delta$ and the rate of decay of the eigenvalues of the covariance operator. If the eigenvalues decay rapidly (e.g., exponentially) or the change is aligned with the first few principal curves of the covariance operator, the dimension reduction approach is preferable as it balances the signal to the noise of those components. However, the loss of information (disregarding everything orthogonal to the chosen subspace) is an undesirable property  of this approach. In particular, when the change is orthogonal to the (contaminated) first principal curves, the approach based on dimension reduction will fail to detect the change. 

To benefit from advantages of both approaches, we propose a \textit{weighted functional} test statistic that can be written similarly in the functional domain as the fully functional statistics but reflects the covariance structure of the data by including the empirical long-run covariance operator $\hat{C}$ and its first eigenvalue $\hat{\lambda}_1$: 
 \begin{align}
        T_{n,\operatorname{WF}}=T_{n,\operatorname{WF}}(\hat{C}) &= \max_{1 \leq k < n }\left\|(\hat{C}+\hat{\lambda}_1\operatorname{Id})^{-\frac 12} \frac{1}{\sqrt{n}} \sum_{i=1}^k (X_i - \bar{X}_n)\right\|.\label{eq_pen_test_max}
        \end{align} 
Here, $\operatorname{Id}$ denotes the identity operator, and for any positive definite, bounded, self-adjoint operator $A:\mathcal{H}\rightarrow \mathcal{H}$, the operator $A^{-1/2}$ is the unique positive definite solution of $A^{-1/2} A A^{-1/2}=\operatorname{Id}$. For further background on linear operators, see e.g. the book of  \cite{kreyszig1991intro}.

The structure of the test statistic mimics the penalization in ridge regression, and we set the regularization parameter to the first eigenvalue of the operator so that the test statistic is scale-invariant (that is, $X$ and $aX$, $a\in \mathbb{R}\setminus \{0\}$ arbitrary, produce the same test statistic and decision).  

\subsubsection{Comparison of the three test statistics}\label{section_comparison}
We may compare the aforementioned test statistics in their theoretical form, based on the true principle curves for illustrational purposes. Let  $\{v_p\}_{p\in\N}$ be the principle curves, i.e.\ the orthonormal basis of eigenfunctions of the true covariance function (in case of i.i.d.\ errors) respectively long-run covariance (for time-series errors), with PCA-scores ${\eta}_{i,p}= \langle X_i, {v}_p \rangle$, $i = 1,\ldots,n $, $p \in \N$, with variance $\lambda_p=\var(\eta_{1,p})$. Then, the test statistic corresponding to the dimension reduction approach with this fixed basis (with reduction to dimension $d$) is equal to
\begin{equation}
    {T}^2_{n,\operatorname{PC}}(v_1,\ldots,v_p) =\max_{k=1,\ldots,n-1} \sum_{p=1}^d \frac{1}{{\lambda}_p} \frac{1}{n} \left[\sum_{i=1}^k \left({\eta}_{i,p} - {\bar{\eta}}_{n,p}\right)\right]^2. \label{eq_pca_pcav}
\end{equation}
On the other hand, the fully functional version can be rewritten via the principal components as well through Parseval's identity as
\begin{equation}
     {T}^2_{n,\operatorname{FF}} =\max_{k=1,\ldots,n-1} \sum_{p=1}^\infty \frac{1}{n} \left[\sum_{i=1}^k \left({\eta}_{i,p} - {\bar{\eta}}_{n,p}\right)\right]^2. \label{eq_ff_pcav}
\end{equation}
While the right hand side is written in terms of PCA-scores, the value of the statistic is independent of the choice of basis. Here, we write it based on the basis provided by the principle curves to highlight the connection with the other two statistics.

Similarly, by Parseval's identity,  we can rewrite the proposed weighted version based on the true long-run covariance operator $C_{\varepsilon}$ (defined in Assumption \ref{ass_covop}  below) with corresponding first eigenvalue $\lambda_1$  in the following way
\begin{align}
        {T}^2_{n,\operatorname{WF}}(C_\varepsilon) &= \max_{1 \leq k < n }\left\|(C_\varepsilon+\lambda_1 \operatorname{Id})^{-\frac 12} \frac{1}{\sqrt{n}} \sum_{i=1}^k (X_i - \bar{X}_n)\right\|^2 \nonumber = \max_{1 \leq k < n } \sum_{p=1}^\infty \left\langle (C_\varepsilon+\lambda_1 \operatorname{Id})^{-\frac 12} \frac{1}{\sqrt{n}} \sum_{i=1}^k (X_i - \bar{X}_n), v_p\right\rangle^2 \nonumber\\
         &= \max_{1 \leq k < n } \sum_{p=1}^\infty \left\langle  \frac{1}{\sqrt{n}} \sum_{i=1}^k (X_i - \bar{X}_n), (C_\varepsilon+\lambda_1 \operatorname{Id})^{-\frac 12}v_p\right\rangle^2 \nonumber=  \max_{1 \leq k < n } \sum_{p=1}^\infty  \frac{1}{\lambda_p+\lambda_1 }\frac 1n\left\langle \sum_{i=1}^k (X_i - \bar{X}_n), v_p\right\rangle^2 \nonumber \\
        &= \max_{1 \leq k < n } \sum_{p=1}^\infty {\frac{1}{{\lambda}_p+{\lambda}_1}} \frac 1n \left(\sum_{i=1}^k \left({\eta}_{i,p} - {\bar{\eta}}_{n,p}\right)\right)^2 \label{eq_pen_pcatvar}.
    \end{align}
This later representation explains why the test statistic $T_{n,\operatorname{WF}}$ is indeed a compromise between $T_{n, \operatorname{PC}}$ and $T_{n, \operatorname{FF}}$. It reflects the weights of the different components (i.e., the noise present in each direction) but at the same time does not require a reduction to a lower dimensional space, which entails the risk of discarding part of the signal. The regularization $\lambda_1 \operatorname{Id}$ cannot be omitted since the covariance operator is not invertible. Whereas $C_\varepsilon^{\frac 12}$ can be written as $C_\varepsilon^{\frac 12} = \sum_{p=1}^\infty \sqrt{\lambda_p}v_p \otimes v_p$, the linear mapping $C_\varepsilon^{- \frac{1}{2}}$ is unbounded and not continuous at any point in its domain.

If only a finite number of eigenvalues are positive, the problem is, in its essence, finite, although possibly highly dimensional, so that such an offset could be avoided in principle. Otherwise, it cannot be avoided without a loss of  information.

\subsection{Null Asymptotics} \label{ssec_null_asymptotics}

Essentially, we can derive
the limit distribution of the weighted functional statistics from the same functional limit theorem that also gives the limit of the fully functional statistics as long as a consistent estimator of the covariance exists.

For $v,w \in L^2(\mathcal{Z})$, the tensor product is denoted by $v \otimes w (h)= \langle v , h \rangle w$ for any $ h \in L^2(\mathcal{Z})$.  Recall, that linear operators on $L^2(\mathcal{Z})$ with the operator norm $\|B\| = \sup_h \|Bh\|/\|h\| < \infty$ form a Banach space. This norm is equal to the highest eigenvalue of the operator. We make the following assumptions:
\begin{ass} \label{ass_covop}
    We assume that $\{\varepsilon_i\}_{i \in \Z}$ is a centered, strictly stationary sequence of random variables with values in $L^2(\mathcal{Z})$ with $\E \|\varepsilon_1\|^{k} < \infty$ for some $k>2$. Denote  $ C_r = \E[\varepsilon_0 \otimes \varepsilon_r]$ and assume that $\sum_{r \in \Z}\| C_r\|<\infty$. We assume that the long-run covariance operator $C_\varepsilon$ is not the zero operator, i.e.\ $C_\varepsilon:=\sum_{r \in \Z} C_r \neq 0$. 
 Denote the spectral decomposition by
   \begin{equation}
       C_\varepsilon = \sum_{p=1}^\infty \lambda_{p,\varepsilon} v_{p,\varepsilon} \otimes v_{p,\varepsilon} \label{eq_lr_covop}
   \end{equation}
   where $\{\lambda_{p,\varepsilon}\}_{p \in \N}$ is the non-decreasing sequence of eigenvalues and $\{v_{p,\varepsilon}\}_{p \in \N}$ is an orthonormal basis of principal curves of $L^2(\mathcal{Z})$. Furthermore, we assume that $\sum_{p=1}^\infty\sum_{h=-\infty}^\infty|\cov(\langle \varepsilon_0, v_{p,\varepsilon} \rangle,\langle \varepsilon_h, v_{p,\varepsilon} \rangle)|<\infty$.
   \end{ass}
\begin{ass}\label{ass_limit_FF}
Let $\{\varepsilon_i\}_{i \in \Z}$ satisfy 
    \begin{align}\label{eq_fclt_ff}\left\{\frac{1}{\sqrt{n}} \sum_{i=1}^{\lfloor tn \rfloor}\varepsilon_i , \ t \in [0,1] \right\} \stackrel{D^H[0,1]}{\Longrightarrow} \{W_{C_{\varepsilon}}(t),t \in [0,1]\},
    \end{align}
  where $\{W_{C_{\varepsilon}}(t),t \in [0,1]\}$ is a $\mathcal{H}$-valued Brownian motion with an associated  covariance operator $C_\varepsilon$ from Assumption~\ref{ass_covop}.  
\end{ass}
Theorem~\ref{thm_ass_FCLT} in Appendix~\ref{sec_ass_fclt} summarizes results of the validity of Assumption~\ref{ass_limit_FF} under various dependency concepts, including independence, $m$-decomposability, and the physical dependence measure.

 In order to derive the limit distribution of the weighted functional statistics (and for the PCA-based dimension reduced statistic), we need to guarantee that the covariance (and corresponding largest eigenvalue) can be estimated consistently.

 \begin{ass}\label{ass_covariance_estimation}
 Let $\hat C$ be a consistent estimator for $C_{\varepsilon}$ as in Assumption~\ref{ass_covop} in the sense of $\|\hat C - C_{\varepsilon}\|=o_P(1)$.       \end{ass}

\begin{thmrep} \label{thm_asymptotics_h0}
Under  $H_0$ and Assumptions~\ref{ass_covop}, \ref{ass_limit_FF}, and \ref{ass_covariance_estimation}, it holds as $n\to\infty$
$$T_{n,\operatorname{WF}}^2(\hat{C}_n)\dto \sup_{0 \leq t \leq 1 }  \sum_{p=1}^\infty {\frac{{\lambda}_{p,\varepsilon}}{{\lambda}_{p,\varepsilon}+{\lambda}_{1,\varepsilon}}} B^2_p(t)$$
where $B_p$ are standard i.i.d. Brownian bridges.
\end{thmrep}

  The limit distribution of the squared fully functional test statistics under the null hypothesis can similarly be expressed as $\sup_{0\le t\le 1}\sum_{p=1}^\infty \lambda_{p, \varepsilon} B^2_p(t)$ with the same notation. Note, however, that the fully functional statistic does not use an estimator $\hat C$, so that Assumption~\ref{ass_covariance_estimation} is not required.
Because $\sum_{p=1}^\infty \lambda_p < \infty$ we also have $\sum_{p=1}^\infty \frac{\lambda_p}{\lambda_p+\lambda_1} \le \frac{1}{\lambda_1}\sum_{p=1}^\infty \lambda_p < \infty$ and both sequences have the same rate of decay.

Clearly, the limit distribution both of the weighted functional and the fully functional statistic are non-pivotal. Thus,  critical values of the test procedures depend on the unknown $C_{\varepsilon}$ and have to be estimated for each instance.  For the fully functional approach, \cite{sharipov2016sequential} propose to use a sequential block bootstrap for dependent functional data, while \cite{aue2018detecting} use a Monte Carlo estimate of the limit distribution based on a kernel type estimator of the long-run covariance. Furthermore, in practice, the infinite sum is replaced by a finite sum as only a finite number of the eigenvalues can be estimated.

The null distribution of the statistics  $T_{n,PC}$ based on dimension reduction via principle components is pivotal and can be found in Theorem 1 of \cite{berkes2009detecting}, Theorem 5.1 in \cite{hormann2010weakly} or Theorem~2.2 of  \cite{horvath2014testing}.

\subsection{Power Considerations} \label{ssec_power_considerations}

We now focus on the behavior of the test statistic under fixed and local alternatives, where we allow for a very general set of alternatives - not limited to the at-most-one-change nor the delayed gradual change alternative.
We consider the following general shape of the time-varying mean in functional data:
\begin{equation}
           X_i= \mu +\Delta_n\, g\left(\frac in \right) + \varepsilon_i, \label{alt_grad_local}
\end{equation}
where the change function $g$ fulfills the following assumption:
\begin{ass} \label{ass_g_funkce}
  Let $g: [0,1] \to \R$ be a non-constant bounded Riemann integrable and piecewise continuous function such that $g(0)=0$, where piecewise continuity means that $[0,1]$ can be partitioned into a finite number of intervals of length $>0$, where $g$ is continuous.
\end{ass}

  Some examples of possible change functions are stated in the following Example~\ref{ex_change_g}.
 \begin{example} \label{ex_change_g}
     Assumption~\ref{ass_g_funkce} is in particular satisfied by the following  important change functions: 
     \begin{enumerate}[(a)]
         \item At most one change: $g_{\operatorname{AMOC},\theta^\ast}(x) = \indi[x>\theta^\ast]$ for some $\theta^\ast \in (0,1)$;
         \item Epidemic change: $g_{\operatorname{EP}}(x) = \indi[\theta_1^\ast < x \le \theta_2^\ast]$ for $0<\theta_1^\ast <\theta_2^\ast <1$;
                 \item Multiple  changes: $g_{\operatorname{MC}}(x)=\sum_{j=1}^{q}a_j\,\indi[\theta_j^\ast < x \le \theta_{j+1}^\ast]$ with $a_1+\ldots+a_q=1$, for some $0<\theta_1^{\ast}<\ldots<\theta_q^{\ast}<\theta_{q+1}^{\ast}=n$, where $q$ is the number of change points.
         \item Delayed gradual change: $g_{\operatorname{DG},\theta^{\ast}}(x)=h(x-\theta^\ast)$ for some smooth $h$ with $h(x)=0$ for $x\le 0$, see also Section~\ref{sec_gradual_change} for a statistic specifically designed for this alternative;
         \item constant linear constant: $g_{\operatorname{CLC}}=\min\{\max\{(x-\theta_1^\ast)/(\theta_2^\ast-\theta_1^\ast),0\},1\}$ for some $0<\theta_1^\ast <\theta_2^\ast <1$.
\end{enumerate}
 \end{example}

Assumption~\ref{ass_g_funkce} implies that there exists an $n_0$ such that for all $n > n_0$ at least one of the equalities in $g(0) = g(\frac 1n) = \cdots = g(\frac{n-1}{n}) = g(1)$ does not hold, such that $H_0$ is distinguishable from $H_1$.
More precisely, by the fundamental theorem of calculus for Riemann integrable functions, Assumption~\ref{ass_g_funkce} implies the existence of a $\theta^\ast$ such that
\begin{equation}
 |G_0(\theta^\star)|>0 \qquad \text{for}\qquad   G_0(\theta)=\int_0^{\theta} g(x) \ dx - \theta \int_0^1 g(x) \ dx.
    \label{theta_ast}
\end{equation}  
The $\theta^\star$ is not unique, e.g.\ for the AMOC example, this holds true for any $0<\theta<1$. In this example, the maximum in absolute value is obtained exactly at the change point, which is why the maximizer of any of the discussed statistics designed for the AMOC situation, is a good estimator for the change point.

Under alternatives, we do not require $\hat{C}_n$ to be consistent towards the covariance operator of the errors $C_{\varepsilon}$. Indeed, for a general alternative as in \eqref{alt_grad_local}, it will typically not be possible to get such an estimator, because it is in general not possible to reconstruct the noise sequence. However, in simple cases with known $g$, such as for the correctly specified AMOC situation this can be achieved, see, e.g., \cite[Section~3.1]{aue2018detecting}. 

Nevertheless, to obtain asymptotic power one of the test, it is sufficient that $\|\hat C_n\|=O_P(1)$  or, more generally, that $\|\hat C_n\|=o_P(n\|\Delta_n\|^2)$, i.e.\ it does not even have to be bounded in probability. For appropriate dependence conditions, the kernel variance estimator fulfills $\|\hat C_n\|=o_P(n\|\Delta_n\|^2)$, see  Lemma \ref{lemma_chat} in the appendix.

 \begin{thm} \label{thm_alternatives}
      Let $\|\hat C_n\|=o_P(n\|\Delta_n\|^2)$ and $\sqrt{n}\|\Delta_n\|\to \infty$, then 
    under Assumptions~\ref{ass_covop}, \ref{ass_limit_FF}, and \ref{ass_g_funkce} it holds $T_{n,WF}(\hat{C}) \stackrel{\mathsf{P}}{\to} \infty$ as $n \to \infty$.
\end{thm}

Consistency of the fully functional statistic and the dimension-reduced statistic can be obtained analogously, where the dimension-reduced statistic can only find changes, where $\Delta_n$ is not orthogonal to the first $d$ eigenfunctions.

\section{Gradual Change} \label{sec_gradual_change}

The AMOC test statistic discussed in Section~\ref{sec_abrupt_change} is consistent against general alternatives as stated in Theorem~\ref{thm_alternatives}.  However, as the statistic is designed with abrupt AMOC alternatives in mind, it may not result in the best power under gradual trends that can often be expected in natural phenomena.
Instead of an abrupt AMOC alternative such observations are better modeled with a gradually changing mean starting at some unknown point in time as in the delayed gradual change model given in \eqref{alt_d_g}, i.e.\ $\mu_i = \mu+ g_{\operatorname{DG},\theta^\ast}(i/n) \cdot\Delta$ as in Example~\ref{ex_change_g}.

For univariate data, \cite{huvskova2000limit,huvskova2002asymptotic} discuss the asymptotic behavior of statistics related to the log-likelihood ratio tests for the delayed gradual change model $g_{\operatorname{DG},\theta^\ast}(x)=h(x-\theta^{\ast})$ with known $h$ but unknown change point $\theta^{\ast}$. We extend these statistics to functional data, allowing for all three approaches as in the previous section. Since in general the true shape of the change function as indicated by $h$ will not be known exactly, we will again analyze the power behavior of these statistics under general alternatives.

\subsection{Test Statistics and Null Asymptotics} \label{ssec_test_statistic}
The test statistics in the previous section were all based on maxima of partial sum processes, which can also be written as weighted sums in the following way
\begin{align*}
\sum_{i=1}^k    (X_i-\bar{X}_n)=-\sum_{i=1}^n\indi[i>k]\, (X_i-\bar{X}_n)=- \sum_{i=1}^ng_{\operatorname{AMOC},\frac kn}\left(\frac i n \right) \, (X_i-\bar{X}_n),
\end{align*}
with $g_{\operatorname{AMOC},\frac kn}$ as in Example~\ref{ex_change_g}.
Similarly,  we propose now test statistics based on sums that are weighted with the change function for the delayed gradual change $g_{\operatorname{GC},k/n}(i/n)=h((i-k)/n)$ for a class of functions $h$ specified in Assumption~\ref{ass_function_h1}. In particular, this class includes the the functions $h(x)=x_+^\alpha$, $\alpha>0$, that were used in \cite{huvskova2000limit}.

To elaborate, the PCA-based dimension reduced statistics for gradual changes is given by
\begin{align}\label{eq_pca_pcav_grad}
    T^2_{n,\operatorname{PC}}(h;\hat{C})=\max_{1\le k< n }
     \sum_{p=1}^d \frac{1}{{\hat\lambda}_p} \frac{1}{n} \left[\sum_{i=1}^nh\left(\frac{i-k}{n}\right)\, \left(\hat{\eta}_{i,p} - {\bar{\hat \eta}}_{n,p}\right)\right]^2,
\end{align}
where $\hat \eta_{i,p}=\langle X_i,\hat{v}_p\rangle $ for the orthonormal eigenfunction $\hat{v}_p$ corresponding to the $p$-th largest eigenvalue $\hat{\lambda}_p$ of $\hat{C}$ are the principle component scores.
Similarly, the (unweighted) fully functional statistic for gradual changes is given by 
 $$T_{n, \operatorname{FF}}(h)=\max_{1\leq k < n}\left\| \frac{1}{\sqrt{n}}\sum_{i=1}^n h\left(\frac{i-k} n\right)\,  ( X_i - \bar{X}_n)\right\|. $$ 
while the weighted fully functional statistic is defined as
 $$T_{n, \operatorname{WF}}(h;\hat{C})=\max_{1\leq k < n} \left\|(\hat C+\hat{\lambda}_1 \operatorname{Id})^{-\frac 1 2}\,\frac{1}{\sqrt{n}}\sum_{i=1}^n h\left(\frac{i-k} n\right)\,  ( X_i - \bar{X}_n)\right\|, $$ 
 where $\hat{\lambda}_1$ is the largest eigenvalue of $\hat C$.

We make the following assumption on the function $h$ used in the above test statistics:
\begin{ass} \label{ass_function_h1}
      Let $h:[-1,1] \to \R$ be a non-constant function with  $h(x)=0$ for $x \leq 0$, which has bounded total variation and which is  $\alpha$-H\"older continuous on $[0,1]$ for some $\alpha>0$,  i.e.\ there exists $\beta>0$
     $$|h(t)-h(s)| \leq \beta |s-t|^\alpha, \quad s,t \in [0,1].$$
\end{ass}
This assumption holds e.g.  for the functions $h(x)=x_+^{\alpha}$, $\alpha>0$, which were proposed by \cite{huvskova2000limit}. 
Assumption~\ref{ass_function_h1} is weaker than the assumptions in \cite{huvskova2000limit, huvskova2002asymptotic}: In particular, because our proof works directly with the bounded total variation, we do no longer require $h$ to be non-decreasing nor absolutely continuous.

We first discuss the limit distribution of the fully functional and the weighted functional test statistic under the null hypothesis. As in Section~\ref{sec_abrupt_change} the corresponding limit distribution are not pivotal but do depend on the long-run covariance operator $C_{\varepsilon}$.

\begin{thm} \label{thm_gc_functional}
  Let Assumptions~\ref{ass_covop}, \ref{ass_limit_FF}, and \ref{ass_function_h1} be fulfilled. Let $G_p(t)=\int_0^{1-t} B_p(1-t-y) \ dh(y)$ with $\{B_p(t): t \in (0,1)\}$, $p=1,2,\ldots$, i.i.d.\ Brownian bridges.
  \begin{enumerate}[(a)]
\item Then, it holds under $H_0$ for the fully functional statistics
\begin{align}
T_{n,\operatorname{FF}}(h)\dto \sup_{0<t<1}\sqrt{\sum_{p=1}^\infty \lambda_{p,\varepsilon} G_p^2(t)} .\label{gc_functional_limit}
\end{align}
\item If additionally Assumption \ref{ass_covariance_estimation} holds, then it holds for the weighted functional statistics under $H_0$
\begin{align*}
T_{n,\operatorname{WF}}(h;\hat{C})\dto \sup_{0<t<1} \sqrt{\sum_{p=1}^\infty\frac{\lambda_{p,\varepsilon} }{\lambda_{1,\varepsilon} +\lambda_{p,\varepsilon} }G_p^2(t)}.
\end{align*}
  \end{enumerate}
\end{thm}

\begin{rem}\label{rem_pca_gradual}
The (finite-dimensional) processes $\mathbf{G}=(G_1,\ldots,G_d)^T$ are centered Gaussian processes with covariance structure 
$$\E \mathbf{G}(s)\mathbf{G}(t)^T=I_d \left(\int_{s \vee t}^1 h(x-s)h(x-t) \ dx - \int_s^1h(x-s) \ dx \int_t^1 h(x-t) \ dx \right), \quad t,s \in [0,1],$$
see Remark 2.5 in \cite{huvskova2000limit}.
\end{rem}

 For the proof of the limit distribution of the dimension-reduced statistic based on PCA, we follow the approach based on the assumptions \eqref{eq_ass_est_grad}  taken by \cite{aston2012detecting}, see their Section~2.3 for conditions when \eqref{eq_ass_est_grad} holds, covering both mixing and physical dependence, as long as the first $d+1$ eigenvalues are distinct. The sign function $s_j$ is required as orthogonal eigenfunctions are only unique up to the sign.

\begin{thm}\label{th_gradual_PCA}
    Let Assumptions~\ref{ass_covop}, \ref{ass_limit_FF}, \ref{ass_covariance_estimation} and \ref{ass_function_h1} be fulfilled. Furthermore, let $\lambda_1>\ldots>\lambda_d>0$ and
    \begin{align}\label{eq_ass_est_grad}
    \|\hat{v}_p-s_p\,v_p\|=o_P(1),\qquad \hat{\lambda}_p-\lambda_p=o_P(1),
    \end{align}
    where $s_j=\sign(\int \hat{v}_p(t)v_p(t)\,dt)$, for $p=1,\ldots,d$. Then,  it holds, under $H_0$,
    \begin{align*}
        T_{n,\operatorname{PC}}(h;\hat{C})\dto \sup_{0\le t\le 1}\sqrt{\sum_{p=1}^dG_p^2(t)},
    \end{align*}
    where $G_p(t)=\int_0^{1-t} B_p(1-t-y) \ dh(y)$ with $\{B_p(t): t \in (0,1)\}$, $p=1,\ldots,d$, i.i.d.\ Brownian bridges is as in Theorem~\ref{thm_gc_functional}.
\end{thm}

\begin{rem}\label{rem2}
   The proof of Theorem~\ref{th_gradual_PCA} also works if  a different basis $v_p$ is being used and we set $\lambda_p=\var(\langle \varepsilon_1,v_p\rangle)$. In this case, the Wiener processes in Proposition \ref{thm_wi_d_D} are typically dependent and its correlation is inherited by the Brownian bridges in the limit of Theorem~\ref{th_gradual_PCA}. Consequently, the limit is no longer pivotal such that e.g.\ resampling methods need to be employed to obtain an asymptotically valid test. This is for example the case, when standard PCA based on the covariance function instead of the long-run covariance function is used for functional observations with time-dependency, see e.g.\ \cite{aston2012evaluating} for such an example  for the corresponding AMOC situation.
\end{rem}

\subsection{Power Behavior} \label{ssec_power_behavior}

Assume again the local gradual model \eqref{alt_grad_local} from Subsection~\ref{ssec_power_considerations}, where the true change function $g$ satisfies Assumption~\ref{ass_g_funkce}. 
We show, that a change is detectable if the signal term is bounded away from zero in an environment of some point $\theta$, where the signal term is given by
\begin{equation}\label{eq:detect}
 \sup_{0\leq t\leq1}\left|\int_0^1 h(x-t)\, g(x) \, dx - \int_0^1g(x) \,dx  \int_0^1 h(x-t)\, dx \right| 
 >0.
\end{equation} 
\begin{prop}\label{prop_alt_grad}
   If $g$ satisfies Assumption~\ref{ass_g_funkce} and $h$ satisfies Assumption~\ref{ass_function_h1} with $\inf \{x: h(x)\neq 0\}=0$, then \eqref{eq:detect} holds.
\end{prop}
 The condition $\inf \{x: h(x)\neq 0\}=0$ guarantees that any $0< \theta<1$ as  in \eqref{alt_d_g} yields a non-constant mean, i.e.\ corresponds to an alternative, with $\theta$ being the change point.

Similarly to Theorem~\ref{thm_alternatives} for the AMOC situation we require $\|\hat C_n\|=o_P(n\|\Delta_n\|^2)$ to obtain asymptotic power one, see also  Lemma \ref{lemma_chat} in the appendix.
First, we state the result for the fully functional statistics.
\begin{thmrep} \label{thm_gradual_alternative}
 Let $\sqrt{n}\|\Delta_n\|\to \infty$ and $X_1,\ldots, X_n$ follow the local gradual model \eqref{alt_grad_local}, where the errors fulfill 
Assumptions \ref{ass_covop} and \ref{ass_limit_FF}. If $h$ satisfies Assumption~\ref{ass_function_h1} and 
 \eqref{eq:detect} holds,
 then $T_{n, \operatorname{FF}}(h) \stackrel{\mathsf{P}}{\to} \infty$. If additionally  $\|\hat C_n\|=o_P(n\|\Delta_n\|^2)$, then  $T_{n, \operatorname{WF}}(h) \stackrel{\mathsf{P}}{\to} \infty$.
\end{thmrep}

Next, we give the result for the dimension reduced statistic. In this case, we can no longer expect the estimator $\hat C_n$ to be consistent for the uncontaminated covariance operator $C_\varepsilon$. However, we can still assume that the eigenfunctions converge, that means for some orthonormal functions $\tilde{v}_j\in L^2(\mathcal{Z})$, $j=1,\ldots,p$, we have 
\begin{equation}\label{eq_ass_est_grad_alt}
    \|\hat{v}_p-\tilde s_p\,\tilde v_p\|=o_P(1) \  \  \text{ for } p=1,\ldots,d
    \end{equation}
where $\tilde s_j=\sign(\int \hat{v}_p(t)\tilde v_p(t)\,dt)$.
For the AMOC and epidemic change alternative and the standard PCA (based on the covariance not long-run covariance estimator) estimator this has been shown in \cite[Lemma~2]{berkes2009detecting} respectively Lemma 2.4 in combination with Theorem 2.1 in \cite{aston2012detecting}.

Because the dimension reduced statistic is based on a projection of the data into a lower dimensional space, asymptotically defined by $\tilde{v}_1,\ldots,\tilde{v}_d$, asymptotic power one is only achieved if the signal projected into that subspace is still sufficiently large. Thus, additionally, we assume
\begin{align}\label{eq_ass_grad_alt_2}
    \sum_{p=1}^d \langle\Delta_n, \tilde{v}_p\rangle^2\ge c\, \|\Delta_n\|^2\quad \text{for some } 0<c\le 1.
\end{align}
Indeed, in this case,  by Parseval's identity, the two quantities are  of the same order.

\begin{thmrep} \label{thm_gradual_alternative_DR}    
 Let  $X_1,\ldots, X_n$ follow  model \eqref{alt_grad_local} with a change $n\|\Delta_n\|^2 \rightarrow \infty$,
 and let \eqref{eq_ass_est_grad_alt} and \eqref{eq_ass_grad_alt_2} hold. Furthermore, let the errors fulfill Assumptions \ref{ass_covop} and \ref{ass_limit_FF} and let $\|\hat{C}\|=o_P(n\|\Delta_n\|^2)$.  If $h$ satisfies Assumption~\ref{ass_function_h1} and  \eqref{eq:detect} holds,  then $T_{n, \operatorname{PC}}(h) \stackrel{\mathsf{P}}{\to} \infty$. 
\end{thmrep}
\section{Conclusions}
For the AMOC situation, we introduce a new fully functional test statistics that takes the covariance structure into account, thus compromising between the strength of the dimension-reduction approach and the unweighted fully functional approach that have previously been compared in the literature.
There is far less literature on the detection of gradual changes in functional data, a gap that is narrowed by the present work as we combine methodology previously proposed for gradual changes in univariate data with the three approaches for functional data, that were already discussed for the AMOC situation. 
Furthermore, the results in this paper indicate that all considered test statistics also have power against different deviations from mean-stationarity of the underlying processes, even if the type of change is misspecified a-priori.

  \section*{Acknowledgements}
The research was supported by German Research Foundation (Deutsche Forschungsgemeinschaft - DFG), project Gradual functional changes, KI~1443/6-1 and WE~5988/5 within the WEAVE cooperation between the DFG and GA\v CR. Furthermore, Hedvika Ranošová acknowledges support from the Charles University Grant Agency, project GAUK 70324.
  
\clearpage
\appendix
\section{Validity of the assumptions}
\subsection{Assumptions on sums of the error process}
\subsubsection{Functional central limit theorems}\label{sec_ass_fclt}

In this section, we provide several sets of dependency conditions for which the functional limit theorem in Assumption~\ref{ass_limit_FF} holds. Several notions of weak dependencies were introduced in the literature, starting with \cite{rosenblatt1956central} and \cite{ibragimov1962some}. We restrict the discussion to the following definitions of dependence of functional time series.

\begin{thm}\label{thm_ass_FCLT} Assumption~\ref{ass_limit_FF} is satisfied, if Assumption \ref{ass_covop} holds in addition to  one of the following conditions hold for an i.i.d. sequence $\{\eta_i\}_{i \in \Z}$:
\begin{enumerate}[(a)]
\item \textbf{(Near epoch dependence)} There exist constants $C_1>0,C_2>1/2$, such that for all $n\in\Z, m\in\N$ we have
\begin{equation*}
\Big(E\Big[\big\|\varepsilon_n-E\big[\varepsilon_n\big|\eta_{n-m},...,\eta_{n+m}\big]\big\|^2\Big]\Big)^{1/2}\leq C_1m^{-C_2}.
\end{equation*}
\item  \textbf{($\boldsymbol{L^m}$-decomposability)} $E\|\varepsilon_0\|^{m}$ for some $m>2$ and $\varepsilon_i=g(\eta_{i},\eta_{i-1},\ldots)$ for  a deterministic function $g:\mathcal{S}^{\infty} \to L^2$, and additionally for independent copies $\{\eta_i^\ast\}_{i \in \Z}$ of $\eta_0^\ast$ independent of $\{\eta_i\}_{i \in \Z}$, it holds for some $m>2, \kappa>m$ with $\varepsilon_{i,\ell}^\ast = g(\eta_i,\ldots,\eta_{i-\ell+1},\eta_{i-\ell}^\ast,\eta_{i-\ell-1}^\ast,\ldots)$,  that
\begin{equation*}
\sum_{\ell=0}^\infty \big(\E \|\varepsilon_0 - \varepsilon_{0,\ell}^\ast\|^m\big)^{1/\kappa}<\infty.
\end{equation*}
\item \textbf{(Physical dependence measure)}  $\varepsilon_i=g(\eta_{i},\eta_{i-1},\ldots)$ for  a deterministic function $g:\mathcal{S}^{\infty} \to L^2$, and additionally for independent copies $\{\eta_i^\ast\}_{i \in \Z}$ of $\eta_0^\ast$ independent of $\{\eta_i\}_{i \in \Z}$ and $\varepsilon_{i,\ell}' = g(\eta_i,\ldots,\eta_{i-\ell+1},\eta_{i-\ell}^\ast,\eta_{i-\ell-1},\eta_{i-\ell-2},\ldots)$,
\begin{equation*}
\sum_{\ell=0}^\infty \big(\E \|\varepsilon_0 - \varepsilon_{0,\ell}'\|^2\big)^{1/2}<\infty.
\end{equation*}
\end{enumerate}

\end{thm}
Part a) is a special case of Theorem 4.6 of  \cite{chen1998central},  b) is Theorem 1.1 of  \cite{berkes2013weak}, and  c) is Corollary 1.3 of  \cite{jirak2013weak}.
\begin{thm}\label{thm_fclt}
Let $\{\varepsilon_i\}_{i \in \Z}$ satisfy Assumption~\ref{ass_limit_FF}. Then under $H_0$
\begin{equation}
     T_{n,FF} \stackrel{d}{\to} \sup_{0<t<1} \|W_{C_\varepsilon}(t)-tW_{C_\varepsilon}(1)\|
\end{equation}
  where $\{W_{C_\varepsilon}(t),t \in [0,1]\}$ is a $\mathcal{H}$-valued Brownian motion with an associated covariance operator $C_\varepsilon$.
\end{thm}
This can be shown by standard arguments; we give a proof for completeness:
\begin{proof} As we are under the null-hypothesis, we have
\begin{equation*}
 T_{n,FF}= \sup_{0<t<1}\bigg\|\frac{1}{\sqrt{n}}\sum_{i=1}^{\lfloor nt \rfloor}\Big(\varepsilon_i - \frac{1}{n}\sum_{i=1}^n\varepsilon_i \Big)\bigg\|
\end{equation*}
By the weak convergence of the partial sum process (Assumption~\ref{ass_limit_FF}) and the continuous mapping theorem, we have
\begin{equation*}
 \sup_{0<t<1}  \bigg\|\frac{1}{\sqrt{n}}\sum_{i=1}^{\lfloor nt \rfloor}\varepsilon_i - \frac{t}{\sqrt{n}}\sum_{i=1}^n\varepsilon_i  \bigg\|\stackrel{d}{\to} \sup_{0<t<1} \|W_C(t)-tW_C(1)\|.
\end{equation*}
 Furthermore,
 \begin{equation*}
\sup_{0<t<1} \bigg\|\frac{1}{\sqrt{n}}\sum_{i=1}^{\lfloor nt \rfloor}\Big(\varepsilon_i - \frac{1}{n}\sum_{i=1}^n\varepsilon_i \Big)-\frac{1}{\sqrt{n}}\Big(\sum_{i=1}^{\lfloor nt \rfloor}\varepsilon_i - t\sum_{i=1}^n\varepsilon_i\Big)\bigg\|=\sup_{0<t<1} \frac{1}{\sqrt{n}} \left|\frac{\lfloor nt \rfloor}{n}-t\right|\bigg\|\sum_{i=1}^n\varepsilon_i \bigg\|=\frac{1}{n^{3/2}}\bigg\|\sum_{i=1}^n\varepsilon_i \bigg\|\stackrel{P}{\to}0,
 \end{equation*}
 also using Assumption~\ref{ass_limit_FF}, which completes the proof.
\end{proof}

\subsubsection{Invariance principles of the projections}\label{app_inv_prin}

\begin{prop} \label{thm_wi_d_D}  Let $\{\varepsilon_i\}_{i \in \Z}$ satisfy Assumption~\ref{ass_limit_FF}. Then for every $d\in\N$, possibly after changing the probability space, there exist processes $\{\mathbf W_{\boldsymbol{\lambda},n}^{(d)}(t):t\ge 1\}$ with $\mathbf{W}_{\boldsymbol{\lambda},n}(t)=(\sqrt{\lambda_1}\,W_{n,1}(t),\ldots,\sqrt{\lambda_d}\,W_{n,d}(t))^T$, where $\{W_{n,p}(\cdot)\}$, $p=1,\ldots,d$, are independent standard Wiener processes, such that
\begin{align}\label{eq_ass_gradual_scores}
    \frac{1}{\sqrt{n}}\max_{1\le k\le n}\left\|\sum_{i=1}^k\boldsymbol{\eta}_{d}(i)-\mathbf{W}_{\boldsymbol{\lambda},n}(k)\right\|\overset{a.s.}{\to} 0,
\end{align}
where $\boldsymbol{\eta}_{d}(i)=(\langle \varepsilon_i,v_1\rangle,\ldots,\langle \varepsilon_i,v_d\rangle)^T $ and $\|\cdot\|$ denotes the Euclidean norm.
\end{prop}

\begin{proof} The mapping $T_d:\mathcal{H}\rightarrow \R^d$ (equipped with the Euclidean norm) with  $T_d(y)= (\langle y,v_1\rangle,\ldots,\langle y,v_d\rangle)^T $  is linear and continuous, so by Assumption~\ref{ass_limit_FF} and the continuous mapping theorem, we have the weak convergence
\begin{equation*}
  \left( \frac{1}{\sqrt{n}}\sum_{i=1}^{\lfloor nt \rfloor}\boldsymbol{\eta}_{d}(i)\right)_{t\in[0,1]}= \left( T_d\Big(\frac{1}{\sqrt{n}}\sum_{i=1}^{\lfloor nt \rfloor}\varepsilon_i\Big)\right)_{t\in[0,1]}
  \stackrel{D^d[0,1]}{\Longrightarrow}
\left( T_d\big(W_{C_\varepsilon}\big)\right)_{t\in[0,1]},
\end{equation*}
where $T_d\big(W_{C_\varepsilon}(t)\big)=(\langle W_{C_\varepsilon}(t),v_1\rangle,\ldots,\langle W_{C_\varepsilon}(t),v_d\rangle)^T, \ t\in[0,1]$, is $d$-dimensional Brownian motion with covariance structure $\cov(\langle W_{C_\varepsilon}(t),v_i\rangle,\langle W_{C_\varepsilon}(t),v_j\rangle)=\langle tC_\varepsilon v_i,v_j \rangle =t\langle \lambda_iv_i,v_j \rangle =t\lambda_i \mathbb{1}_{i=j}$. As the correlation between the components is $0$, the Wiener processes $(\langle W_{C_\varepsilon}(t),v_1\rangle)_{t\in[0,1]}$, ..., $(\langle W_{C_\varepsilon}(t),v_d\rangle)_{t\in[0,1]}$ are independent. 

Now by the Skorokhod representation theorem, there exists copies $W_n=(W_n^{(1)},...,W_n^{(d)})^T$ of $T_d\big(W_{C_\varepsilon}(t)\big)$ such that the Skorokhod distance of $\big( \frac{1}{\sqrt{n}}\sum_{i=1}^{\lfloor nt \rfloor}\boldsymbol{\eta}_{d}(1)\big)$, ${t\in[0,1]}$ and $W_n(t)$, $t\in[0,1]$ converges to 0. Setting $W_{n,j}(k)=\sqrt{n}\lambda_j^{-1/2}W_n^{(j)}(k/n)$ for $j=1,..,d$, the statement of the proposition follows, because the maximum is a continuous mapping on $D^d[0,1]$.
\end{proof}

\subsection{Covariance estimation}\label{sec_cov}
In this section, we summarize results under which we get a consistent estimation and even corresponding rates for the covariance operator as required by Assumption~\ref{ass_covariance_estimation}.

To elaborate, we work with a kernel-based estimator of the long-run covariance
\begin{equation}
    \hat{C} = \sum_{r = -n}^n K \left(\frac{r}{h_n}\right) \hat{C}_r \label{eq_kernel_est} 
\end{equation}
where
\begin{equation}
 \hat{C}_r =  
 \begin{cases} 
 \frac{1}{n} \sum_{i=1}^{n-r} (X_i -\bar{X}_n)\otimes(X_{i+r}-\bar{X}_n), \quad r \geq 0,  \\
 \frac{1}{n} \sum_{i=1}^{n-|r|} (X_{i-r} -\bar{X}_n)\otimes(X_{i}-\bar{X}_n), \quad r < 0, 
 \end{cases} \label{eq_lag_cov}
\end{equation}
 with prespecified kernel function $K$ and bandwidth $h_n$. The assumptions on the kernel function and the bandwidth are slightly modified from \cite{horvath2024change} (Assumptions 3.1.4 and 3.1.5).
  For kernels, we thus assume that $K(0)=1$, $K$ is symmetric, bounded and Lipschitz continuous. Moreover, it is zero outside an interval $[-c,c]$ for some $c >0$. For $h_n$ we assume that 
 \begin{equation}
     h_n \to \infty, \quad \frac{h_n}{\sqrt{n}} \to 0 \quad \text{as} \quad  n \to \infty. \label{eq_ass_hn}
 \end{equation}
Assumption 3.1.4 in \cite{horvath2024change} replaces the second convergence by $\frac{h_n}{n} \to 0$ (which is implied by \eqref{eq_ass_hn}). The optimal rate of $h_n$ is $o(n^{\frac{1}{2q+1}})$, see \cite{rice2017plug}, where $q$ is the order of the kernel $K$, that is, a number such that
$$0< \lim_{x \to 0} \frac{1-K(x)}{|x|^q}<\infty.$$ The assumptions in \eqref{eq_ass_hn} are satisfied with the optimal bandwidth rate if $q>\frac{1}{2}$, which is satisfied by all popular kernels mentioned in \cite{horvath2024change}; including the Bartlett kernel $K_B(x)=\max\{1-|x|,0\}$.

The following lemma 
explains the behavior of the kernel-based estimator under the null-hypothesis and under alternatives.

\begin{lemmarep}\label{lemma_chat} Let $K_B$ a symmetric, continuous kernel function, bounded by 1 with compact support and with $K_B(0)=1$. Let the bandwidth fulfill \eqref{eq_ass_hn}.
Consider $L^{m}$-decomposable functional time series, in the sense of Theorem~\ref{thm_ass_FCLT} b) with $\lim_{m\rightarrow \infty} m (\E \|\varepsilon_0 - \varepsilon_{0,\ell}^\ast\|^2)^{1/2}=0$,
fulfilling Assumptions~\ref{ass_covop}.  
\begin{enumerate}[(a)]
\item Then Assumption \ref{ass_covariance_estimation} holds under the null hypothesis. \item Under the alternative model  \eqref{alt_grad_local} with $\sqrt{n}\|\Delta_n\|\rightarrow \infty$, we have that $\|\hat{C}_n\|=o_P(n\|\Delta_n\|^2)$.
\end{enumerate}
\end{lemmarep}

\begin{proof} For Assertion (a), first note that for any $g\in L^2(\mathcal{Z})$, we have $C_\varepsilon(g)=\int c_\varepsilon(s,t)g(t)d\nu(t)$ with $c$ given by $c_\varepsilon(s,t)=\sum_{p=1}^\infty \lambda_{p,\varepsilon}v_{p,\varepsilon}(s)v_{p,\varepsilon}(t)$. Furthermore, we have $\hat{C}(g)=\int \hat{c}_ng(t)d\nu(t)$ with $\hat{c}_n$ as defined by \cite{horvath2013estimation}. If $\int g^2 d\nu\leq 1$, we have
\begin{multline*}
\int \big(\big(\hat{C}-C_\varepsilon\big)(g)\big)^2d\nu=\int \bigg(\int \big(\hat{c}_n(s,t)-c_\varepsilon(s,t)\big)g(t)d\nu(t)\bigg)^2 d\nu(s)\\
\leq \int \bigg(\int \big(\hat{c}_n(s,t)-c_\varepsilon(s,t)\big)^2d\nu(t) \int g^2(t)d\nu(t)\bigg) d\nu(s)\leq \iint \big(\hat{c}_n(s,t)-c_\varepsilon(s,t)\big)^2d\nu(t) d\nu(s)
\end{multline*}
by the H\"older inequality. So we can conclude that
\begin{equation*}
\|\hat{C}-C_\varepsilon\|\leq \iint \big(\hat{c}_n(s,t)-c_\varepsilon(s,t)\big)^2d\nu(t) d\nu(s)\rightarrow 0
\end{equation*}
in probability as $n\rightarrow \infty$, by Theorem 2 by \cite{horvath2013estimation}. The conditions of Theorem~\ref{thm_ass_FCLT} b) ($L^{m}$-decomposability with $\kappa>m>2$) imply the conditions of Theorem 2 by  \cite{horvath2013estimation} (defined like $L^{m}$-decomposability but with $\kappa=m=2$).

For the result in (b), under the alternative,  rewrite $\hat{C}_r$, the lagged covariance operator \eqref{eq_lag_cov},  $r\geq0$, used in the kernel estimate \eqref{eq_kernel_est}, where we can assume without loss of generality, that $\mu=0$ in \eqref{alt_grad_local}:
\begin{align*}
\hat{C}_r &=\frac{1}{n} \sum_{i=1}^{n-r} (X_i -\bar{X}_n)\otimes(X_{i+r}-\bar{X}_n) \\
 &=\frac{1}{n} \sum_{i=1}^{n-r} \left[ (g_i -\bar{g})\Delta_n +\varepsilon_i-\bar{\varepsilon}_n \right]\otimes\left[ (g_{i+r} -\bar{g})\Delta_n +\varepsilon_{i+r}-\bar{\varepsilon}_n \right] \\
 &= \frac{1}{n} \sum_{i=1}^{n-r} (\varepsilon_i-\bar{\varepsilon}_n) \otimes(\varepsilon_{i+r}-\bar{\varepsilon}_n) + \Delta_n\otimes\Delta_n \left[\frac{1}{n} \sum_{i=1}^{n-r} (g_i-\bar{g}_n)(g_{i+r}-\bar{g}_n)\right] \\
 &+ \Delta_n \otimes \left[\frac{1}{n}\sum_{i=1}^{n-r} (g_i-\bar{g}_n)(\varepsilon_{i+r}-\bar{\varepsilon}_n)\right] + \left[\frac{1}{n}\sum_{i=1}^{n-r} (g_{i+r}-\bar{g}_n)(\varepsilon_{i}-\bar{\varepsilon}_n)\right]\otimes\Delta_n \\
 &= \hat{C}_{r, \varepsilon} + \hat{C}_r^{(1)} +\hat{C}_r^{(2)}+\hat{C}_r^{(3)}
\end{align*}
and in the same way for $r<0$
\begin{align*}
\hat{C}_r &=\frac{1}{n} \sum_{i=1}^{n-|r|}(\varepsilon_{i-r}-\bar{\varepsilon}_n) \otimes(\varepsilon_{i}-\bar{\varepsilon}_n) + \Delta_n\otimes\Delta_n \left[\frac{1}{n} \sum_{i=1}^{n-|r|} (g_{i-r}-\bar{g}_n)(g_{i}-\bar{g}_n)\right] \\
 &+ \Delta_n \otimes \left[\frac{1}{n}\sum_{i=1}^{n-|r|} (g_{i-r}-\bar{g}_n)(\varepsilon_{i}-\bar{\varepsilon}_n)\right] + \left[\frac{1}{n}\sum_{i=1}^{n-|r|} (g_{i}-\bar{g}_n)(\varepsilon_{i-r}-\bar{\varepsilon}_n)\right]\otimes\Delta_n \\
 &= \hat{C}_{r, \varepsilon} + \hat{C}_r^{(1)} +\hat{C}_r^{(2)}+\hat{C}_r^{(3)}
\end{align*}
and denote the summed versions $\hat{C}^{(j)} =  \sum_{r = -n}^n K \left(\frac{r}{h_n}\right) \hat{C}^{(j)}_{r}$ for $j = 1,2,3$. For the estimator based on the residuals $\hat{C}_{r, \varepsilon}$, we know by the first part of the lemma that $ \sum_{r = -n}^n K \left(\frac{r}{h_n}\right) \hat{C}_{r, \varepsilon}$ converges to $C_{\varepsilon}$ in probability, so 
\begin{equation*}
\sum_{r = -n}^n K \left(\frac{r}{h_n}\right) \hat{C}_{r, \varepsilon}=O_P(1)=o_P(n\|\Delta_n\|^2)
\end{equation*}
as $\sqrt{n}\|\Delta_n\|\rightarrow \infty$. Now for $\hat{C}^{(1)}$, which is deterministic, we have by our Assumption~\ref{ass_g_funkce} ($g$ being bounded and piecewise continous), that
\begin{equation*}
\frac{1}{n} \sum_{i=1}^{n-r} (g_i-\bar{g}_n)(g_{i+r}-\bar{g}_n)\leq \frac{1}{n}\sum_{i=1}^n(g_i-\bar{g}_n)^2\rightarrow \int_0^1 \left(g(x)  - \int_0^1 g(y) \, dy \right)^2\, dx
\end{equation*}
as $n\rightarrow\infty$, so  $\|\hat{C}_r^{(1)}\| =O(\|\Delta_n^2\|)$. As $K_B$ has a bounded support and thus we have $O(h_n)$ summands, it follows that  $\|\hat{C}^{(1)}\| =O(h_n\|\Delta_n\|^2)=o(n\|\Delta_n\|^2)$ as $h_n=o(\sqrt{n})$. For $\hat{C}^{(2)}$, recall that $g$ is bounded and that $E[|\|\varepsilon_1\|]<\infty$, so
\begin{equation*}
E\bigg[\Big\|\frac{1}{n}\sum_{i=1}^{n-r} (g_{i+r}-\bar{g}_n)(\varepsilon_{i}-\bar{\varepsilon}_n)\Big\|\Bigg]\leq C<\infty
\end{equation*}
for some constant $C$ not depending on $r$, and thus $\|\hat{C}_r^{(2)} \|=O_P(\|\Delta_n\|)$. We conclude that $\|\hat{C}^{(2)}\| =O_P(h_n\|\Delta_n\|)=o_P(n\|\Delta_n\|^2)$. With the same arguments also $\|\hat{C}^{(3)}\| =O_P(h_n\|\Delta_n\|)=o_P(n\|\Delta_n\|^2)$, which completes the proof.
\end{proof}
 
\section{Proofs}

\subsection{Proofs of Section~\ref{sec_abrupt_change}}
We start with a proof for the null asymptotics of the weighted functional statistics as given in Theorem~\ref{thm_asymptotics_h0}.

\begin{proof}[Proof of Theorem~\ref{thm_asymptotics_h0}]
 Denote ${T}_{n,WF}(C_{\varepsilon})$ the version of the test statistic $T_{n,WF}$ with the true long-run covariance operator $C_\varepsilon$ and the true corresponding largest eigenvalue $\lambda_{1,\varepsilon}$. 
 Under the null hypothesis, 
it holds by Assumption~\ref{ass_limit_FF} as $n\to\infty$
    \begin{equation}
      {T}_{n,WF}(C_{\varepsilon})\dto  \sup_{0 \leq t \leq 1 }\left\|(C_\varepsilon+\lambda_{1,\varepsilon} \operatorname{Id})^{-\frac 12} (W_C(t)-tW_C(1))\right\|=\sup_{0 \leq t \leq 1 }  \sum_{p=1}^\infty {\frac{{\lambda}_{p,\varepsilon}}{{\lambda}_{p,\varepsilon}+{\lambda}_{1,\varepsilon}}} B^2_p(t) \label{limita_max}
    \end{equation}
    where $W_C(t)$ is a Wiener process associated with a covariance operator $C_{\varepsilon}$ and $B_p$ are standard i.i.d.\ Brownian bridges.

The operator $(C_\varepsilon+\lambda_{1,\varepsilon} \operatorname{Id})$ clearly shares the same eigenfunctions with $C_{\varepsilon}$ with corresponding eigenvalues $\lambda_{1,\varepsilon}+\lambda_{p,\varepsilon}$. Consequently the operator $(C_\varepsilon+\lambda_{1,\varepsilon} \operatorname{Id})^{\frac 12}$ also shares the same eigenfunctions with corresponding eigenvalues $\sqrt{\lambda_{1,\varepsilon}+\lambda_{p,\varepsilon}}$ (see e.g. \cite{brayman2024functional}, Chapter 12), such that \begin{align*}
&(C_\varepsilon+\lambda_{1,\varepsilon} \operatorname{Id})^{-\frac 12} v_{p,\varepsilon}
=(C_\varepsilon+\lambda_{1,\varepsilon} \operatorname{Id})^{-\frac 12} (\lambda_{p,\varepsilon}+\lambda_{1,\varepsilon})^{-\frac 12}(C_\varepsilon+\lambda_{1,\varepsilon} \operatorname{Id})^{\frac 12}v_{p,\varepsilon}\\
&=(\lambda_{p,\varepsilon}+\lambda_{1,\varepsilon})^{-\frac 12}(C_\varepsilon+\lambda_{1,\varepsilon} \operatorname{Id})^{-\frac 12}(C_\varepsilon+\lambda_{1,\varepsilon} \operatorname{Id})^{\frac 12}v_{p,\varepsilon}= (\lambda_{p,\varepsilon}+\lambda_{1,\varepsilon})^{-\frac 12} v_{p,\varepsilon}.
\end{align*}
As $(\lambda_{p,\varepsilon}+\lambda_{1,\varepsilon})^{-\frac 12} \nearrow \lambda_{1,\varepsilon}^{- \frac 12}$ as $p\rightarrow \infty$, we get
 $ \| (C_\varepsilon+\lambda_{1,\varepsilon} \operatorname{Id})^{-\frac 12} \|  = \lambda_{1,\varepsilon}^{- \frac 12}$. Analogously, it holds 
 \begin{align}
 \|(\hat{C}+\hat{\lambda}_1\operatorname{Id})^{-\frac 12} \| = \hat{\lambda}_1^{-\frac{1}{2}}.\label{eq_value_inversesquarerooteigenvalue}
 \end{align}
As  $\|\hat{C}- C_\varepsilon\| = o_p(1)$, by Assumption~\ref{ass_covariance_estimation},  for the largest eigenvalues, we have $|\hat{\lambda}_1-\lambda_{1,\varepsilon}|=\big|\|\hat{C}\|-\|C_\varepsilon\|\big|\le \|\hat{C}- C_\varepsilon\| = o_p(1)$. Consequently, 
\begin{align}\label{eq_proof_21_1}
\|(\hat{C}+\hat{\lambda}_1 \operatorname{Id}) - (C_\varepsilon+\lambda_{1,\varepsilon} \operatorname{Id})\| = o_p(1)\quad \text{ and }\quad\hat{\lambda}_1^{-1/2}=O_P(1).
\end{align}

The functional $A\rightarrow A^{1/2}$ is operator concave (this follows analogously as for matrices, see Chapter 5 of \cite{bhatia2013matrix}), so we have that 
\begin{align}
&\Big\|\big(\hat{C}+\hat{\lambda}_1 \operatorname{Id}\big)^{1/2} - \big(C_\varepsilon+\lambda_{1,\varepsilon} \operatorname{Id}\big)^{1/2}\Big\| \leq \| |(\hat{C}+\hat{\lambda}_1 \operatorname{Id}) - (C_\varepsilon+\lambda_{1,\varepsilon} \operatorname{Id})|^{1/2}\|\notag\\
&= \| |(\hat{C}+\hat{\lambda}_1 \operatorname{Id}) - (C_\varepsilon+\lambda_{1,\varepsilon} \operatorname{Id})|\|^{1/2}= o_p(1), \label{eq_consistency_cov_square_roots}
\end{align}
where the absolute value $|A|$ of an operator $A$ can defined by replacing the eigenvalues by their absolute values.

Denote $S(t)= \frac{1}{\sqrt{n}} \sum_{i=1}^{\lfloor tn \rfloor} (X_i - \bar{X}_n) = \frac{1}{\sqrt{n}}(\sum_{i=1}^{\lfloor tn \rfloor} X_i - \frac{\lfloor tn \rfloor}{n}\sum_{i=1}^n X_i)$
and by \eqref{limita_max} we get $\sup_{0<t<1}\|(C_\varepsilon+\lambda_{1,\varepsilon} \operatorname{Id})^{-\frac{1}{2}}S(t)\| = O_P(1)$. By the inverse triangle inequality and \eqref{eq_proof_21_1}, we now get 
\begin{align*}
&|{T}_{n,WF}(\hat{C})-{T}_{n,WF}(C_{\varepsilon})| = \left|\sup_{0<t<1}\|(\hat{C}+\hat{\lambda}_1 \operatorname{Id})^{-\frac{1}{2}}S(t)\|- \sup_{0<t<1}\|(C_\varepsilon+\lambda_{1,\varepsilon} \operatorname{Id})^{-\frac{1}{2}}S(t)\|\right| \\
&\leq \sup_{0<t<1} \left\|\left((\hat{C}+\hat{\lambda}_1 \operatorname{Id})^{-\frac 12 }-(C_\varepsilon+\lambda_{1,\varepsilon} \operatorname{Id})^{-\frac 12} \right)S(t) \right\| \\
&=  \sup_{0<t<1} \left\|(\hat{C}+\hat{\lambda}_1 \operatorname{Id})^{-\frac 12 }\left((C_\varepsilon+\lambda_{1,\varepsilon} \operatorname{Id})^{\frac 12}-(\hat{C}+\hat{\lambda}_1 \operatorname{Id})^{\frac 12 } \right)(C_\varepsilon+\lambda_{1,\varepsilon} \operatorname{Id})^{-\frac 12}S(t) \right\| \\
&\leq \|(\hat{C}+\hat{\lambda}_1 \operatorname{Id})^{-\frac 12 }\| \left\|(\hat{C}+\hat{\lambda}_1 \operatorname{Id})^{\frac 12 } -(C_\varepsilon+\lambda_{1,\varepsilon} \operatorname{Id})^{\frac 12}\right\|   \sup_{0<t<1} \|(C_\varepsilon+\lambda_{1,\varepsilon} \operatorname{Id})^{-\frac 12}S(t)\| \\
&= \hat{\lambda}_1^{-\frac 12} \left\|(\hat{C}+\hat{\lambda}_1 \operatorname{Id})^{\frac 12 } -(C_\varepsilon+\lambda_{1,\varepsilon} \operatorname{Id})^{\frac 12}\right\|  \cdot {T}_{n,WF}(C_{\varepsilon})=   o_P(1).
\end{align*}
\end{proof}

\begin{proof}[Proof of Theorem~\ref{thm_alternatives}]
It follows from  the submultiplicativity of the induced norm, that $\|A^{-1} x\| \geq \|A\|^{-1} \|x\|$.
  Thus, under the alternative  \eqref{alt_grad_local} 
    it holds with $G$ and $\theta^{\star}$ as in \eqref{theta_ast}
    \begin{align}
       & T_{n,WF}(\hat{C}_n) = 
          \max_{1 \leq k <n }\bigg\|(\hat{C}_n+\hat{\lambda}_{1,n}\operatorname{Id})^{- \frac 1 2} \frac{1}{\sqrt{n}}
          \bigg[\Delta_n\bigg(\sum_{i=1}^kg\left(\frac in\right) -\frac{k}{n} \sum_{i=1}^{n} g\left(\frac{i}{n}\right)\bigg)+\Big(\sum_{i=1}^k\varepsilon_i -\frac{k}{n} \sum_{i=1}^{n} \varepsilon_i\Big)\bigg]
          \bigg\|\notag \\
        &\ge \big\|(\hat{C}_n+\hat{\lambda}_{1,n}\operatorname{Id})^{\frac 1 2} \big\|^{-1}  \bigg[\bigg\|\sqrt{n}\Delta_n\bigg(\sum_{i=1}^{\lfloor n \theta^{\star}\rfloor}g\!\left(\frac in\right)\frac{1}{n} -\frac{\lfloor n \theta^{\star}\rfloor}{n} \sum_{i=1}^{n} g\!\left(\frac{i}{n}\right)\frac{1}{n}\bigg)\bigg\|-\max_{1\le k\le n}\bigg\|\frac{1}{\sqrt{n}}\Big(\sum_{i=1}^k\varepsilon_i -\frac{k}{n} \sum_{i=1}^{n}\!\varepsilon_i\Big)\bigg\|\bigg]\notag\\
&=\big\|(\hat{C}_n+\hat{\lambda}_{1,n}\operatorname{Id})^{\frac 1 2} \big\|^{-1} \,\left[ \left\|\sqrt{n}\Delta_n\right\|\, (G(\theta^\star)+o(1))
+O_P(1)
        \right],\label{eq_proof_alt}
    \end{align}
    where the $O_P(1)$-term follows from Assumption~\ref{ass_limit_FF} and the convergence of the sum to  $G$ follows from Riemanns integrability of $g$.

It holds $\left\|(\hat{C}+\hat{\lambda}_1\operatorname{Id})^{1/2} \right\| = \sqrt{2 \hat{\lambda}_1} = \sqrt{2\|\hat{C}\|}$, such that by assumption $    \left\|(\hat{C}_n+\hat{\lambda}_{1,n}\operatorname{Id})^{1/2} \right\| \left\|\sqrt{n}\Delta_n\right\|^{-1}=o_P(1)$, which implies that $T_{n,WF}(\hat{C}_n)$ converges to $\infty$ in probability in light of \eqref{eq_proof_alt}.

\end{proof}

\subsection{Proofs of Section~\ref{sec_gradual_change}}
We first prove the following auxiliary lemma.

\begin{lemmarep} \label{thm_gc_d_D}
Let $\boldsymbol{\eta}(i)=\boldsymbol{\eta}_d(i)$ fulfill \eqref{eq_ass_gradual_scores} (for an arbitrary multivariate time series $\{\boldsymbol{\eta}\}$, that is not necessarily obtained as a projection of functional data) and let $h$ satisfy Assumption~\ref{ass_function_h1}.
Then, there exist independent standard Brownian bridges $\{B_{n,j}(t):0\le t\le 1\}$, $j=1,\ldots,d$, such that
\begin{equation}
     \sup_{0<t<1} \left\| \frac{1}{\sqrt{n}} \sum_{i=1}^n h\left(\frac{i-\lfloor nt\rfloor }{n}\right)(\boldsymbol{\eta}(i)-\bar{\boldsymbol{\eta}}_n) - \mathbf{G}_{\boldsymbol{\lambda},n}\left(t\right)  \right\|   \stackrel{\mathsf{P}}{\to} 0 \label{eq_wi_weighted_d_D}
\end{equation} 
with $\mathbf{G}_{\boldsymbol{\lambda},n}(t) = (\sqrt{\lambda_1}\,G_{n,1}(t), \ldots, \sqrt{\lambda_d}\, G_{n,d}(t))'$, $0\le t\le 1$, and 
$G_{n,p} (t) = \int_0^{1-t}  B_{n,p}(1-t-y) \ dh(y).$ 
\end{lemmarep}

\begin{proof}
The proof follows closely the univariate case in Theorem 2.1 of \cite{huvskova2000limit}.
Setting $s=1-t$ and $j=n-i+1$, \eqref{eq_wi_weighted_d_D} is equivalent to
\begin{align}\label{eq_reformulation_lemma}
  &  \sup_{0<s<1}\bigg\| \frac{1}{\sqrt{n}} \sum_{j=1}^n h\left(\frac{\lceil ns \rceil-j+1}{n}\right)(\boldsymbol{\eta}(n-j+1)-\bar{\boldsymbol{\eta}}_n) - \tilde{\mathbf{G}}_{\boldsymbol{\lambda},n}\left(s\right)  \bigg\|   \stackrel{\mathsf{P}}{\to} 0,
\end{align}
where  $\tilde{\mathbf{G}}_{\boldsymbol{\lambda},n} = (\sqrt{\lambda_1}\,\tilde G_{n,1}, \ldots, \sqrt{\lambda_d}\,\tilde G_{n,d})'$ with
$\tilde G_{n,p} (s)=G_{n,p}(1-s) = \int_0^{s} B_{n,p}(s-y) \ dh(y). $ 
Since
\begin{align*}
    \sum_{j=1}^k(\boldsymbol{\eta}(n-j+1)-\bar{\boldsymbol{\eta}}_n)=-\sum_{j=k+1}^n(\boldsymbol{\eta}(n-j+1)-\bar{\boldsymbol{\eta}}_n)=-\sum_{i=1}^{n-k}(\boldsymbol{\eta}(i)-\bar{\boldsymbol{\eta}}_n)
\end{align*}
and $B_{n,p}(s):=- \frac{1}{\sqrt{\lambda_p\,n}} (W_{n,p}(n (1-s))-(1-s)W_{n,p}(n))$ are standard Brownian bridges, it follows from \eqref{eq_ass_gradual_scores} that
\begin{align*}
    \max_{1\le k\le n}\left\|\frac{1}{\sqrt{n}}\sum_{j=1}^k(\boldsymbol{\eta}(n-j+1)-\bar{\boldsymbol{\eta}}_n)
    -\mathbf{B}_{\boldsymbol{\lambda},n}\left(\frac k n\right)
    \right\|\overset{P}{\to} 0,
\end{align*}
where $\mathbf{B}_{\boldsymbol{\lambda},n}(s)=(\sqrt{\lambda_1}\,B_{n,1}(s),\ldots,\sqrt{\lambda_d}\,B_{n,d}(s))^T.$
Consequently, by a brief calculation, recalling $h(x)=0$ for $x\le 0$,
\begin{align}
& \sup_{s\in[0,1]}\bigg\|\sum_{j=1}^n h\Big(\frac{\lceil ns \rceil-j+1}{n}\Big)(\boldsymbol{\eta}(n-j+1)-\bar{\boldsymbol{\eta}}_n)-\sum_{i=1}^{\lceil ns \rceil}\Big(h\Big(\frac{i}{n}\Big)-h\Big(\frac{i-1}{n}\Big)\Big)\sqrt{n}\,\mathbf{B}_{\boldsymbol{\lambda},n}\Big(\frac {\lceil ns \rceil -i+1}{n}\Big)\bigg\|   \label{proofB2sum}\\
=&\sup_{s\in[0,1]}\bigg\|\sum_{i=1}^{\lceil ns \rceil}\Big(h\Big(\frac{i}{n}\Big)-h\Big(\frac{i-1}{n}\Big)\Big)\bigg(\sum_{j=1}^{\lceil ns \rceil-i+1} (\boldsymbol{\eta}(n-j+1)-\bar{\boldsymbol{\eta}}_n)-\sqrt{n}\,\mathbf{B}_{\boldsymbol{\lambda},n}\Big(\frac {\lceil ns \rceil -i+1}{n}\Big)\bigg)\bigg\|\nonumber  \\
\leq& \max_{1\le k\le n}\sum_{i=1}^{k}\Big|h\Big(\frac{i}{n}\Big)-h\Big(\frac{i-1}{n}\Big)\Big|\max_{1\le k\le n}\left\|\sum_{j=1}^k(\boldsymbol{\eta}(n-j+1)-\bar{\boldsymbol{\eta}}_n)    -\sqrt{n}\,\mathbf{B}_{\boldsymbol{\lambda},n}\left(\frac k n\right)\right\|=o_P\left(\sqrt{n}\right) \nonumber 
\end{align}
as $\sum_{i=1}^{k}\big|h(i/n)-h((i-1)/n)\big|\leq \operatorname{T\!V}(h)<\infty$ (where $\operatorname{T\!V}(h)$ denotes the total variation of $h$) by Assumption~\ref{ass_function_h1}.
We first study the summands in \eqref{proofB2sum} form $i=1$ to $\lfloor ns \rfloor$, and will consider the possible remaining summand for $i=\lceil ns \rceil$ later.
\begin{align*}
 &   \sum_{i=1}^{\lfloor n s\rfloor}\left(h\left(\frac{i}{n}\right)-h\left(\frac{i-1}{n}\right)\right)\sqrt{n}\,B_{n,p}\left(\frac {\lceil n s \rceil -i+1}{n}\right)
 =\sum_{i=1}^{\lfloor ns \rfloor}\int_{(i-1)/n}^{i/n}\sqrt{n}\,B_{n,p}\left(\frac {\lceil n s \rceil -i+1}{n}\right)\,dh(y)
 \\
& =\sum_{i=1}^{\lfloor ns \rfloor}\int_{(i-1)/n}^{i/n}\sqrt{n}\,B_{n,p}(s-y)\,dh(y)
 +\sum_{i=1}^{\lfloor ns\rfloor}\int_{(i-1)/n}^{i/n}\sqrt{n}\,\Big(B_{n,p}\Big(\frac {\lceil n s \rceil -i+1}{n}\Big)-B_{n,p}(s-y)\Big)\,dh(y) \\
&=\sqrt{n}\int_0^s\,B_{n,p}(s-y)\,dh(y)
+\int_{\lfloor ns\rfloor/n}^s \sqrt{n}\left(B_{n,p}(s-y)-B_{n,p}(0)\right)\,dh(y)\\
&\qquad +O(1)\, \left(\sqrt{n}\,\sup_{0\le h\le 3/n}\sup_{0\le x\le 1-h}|B_{n,p}(x+h)-B_{n,p}(x)|\right)\,
\int_0^{\lfloor n s\rfloor/n}\,\,dh(y)
\\
&=\sqrt{n}\,\int_0^s\,B_{n,p}(s-y)\,dh(y) +o_P(\sqrt{n}),
\end{align*}
where the last line follows  from $|\int_0^{\lfloor n s\rfloor/n}\,\,dh(y)|\le \operatorname{T\!V}(h)<\infty$ as well as
Theorem 1.4.1 in \cite{csorgo1981strong}, which implies \\$\sqrt{n}\,\sup_{0\le h\le 3/n}\sup_{0\le x\le 1-h}|B_{n,p}(x+h)-B_{n,p}(x)|=O_P(\sqrt{\log n})=o_P(\sqrt{n})$ and thus also
\begin{align*}
    &\int_{\lfloor ns\rfloor/n}^s \sqrt{n}\left(B_{n,p}(s-y)-B_{n,p}(0)\right)\,dh(y)=\operatorname{T\!V}(h)\,O_P\left(\sqrt{\log n}\right)=o_P(\sqrt{n})\\
    &\text{and }\left(h\left(\frac{\lceil ns\rceil}{n}\right)-h\left(\frac{\lceil ns\rceil-1}{n}\right)\right)\sqrt{n}\,\left(B_{n,p}\left(\frac {1}{n}\right)-B_{n,p}(0)\right) =\operatorname{T\!V}(h)\,O_P\left(\sqrt{\log n}\right)=o_P(\sqrt{n}).
\end{align*}
The latter expression corresponds to the summand for $i=\lceil ns \rceil$ in \eqref{proofB2sum}, thus completing the proof of \eqref{eq_reformulation_lemma}.
\end{proof}

\begin{proof}[Proof of Theorem~\ref{thm_gc_functional}]
As there is no change  we have $X_i - \bar{X}_n = \varepsilon_i -\bar{\varepsilon}_n= \sum_{p=1}^\infty \langle \varepsilon_i - \bar{\varepsilon}_n, v_{p,\varepsilon} \rangle\, v_{p,\varepsilon}$ with $v_{p,\varepsilon}$ as in Assumption~\ref{ass_covop}. For arbitrary $L \in \N$ define 
\begin{align*}
    X_i^L &=  \sum_{p=1}^L \langle \varepsilon_i, v_{p,\varepsilon} \rangle\, v_{p,\varepsilon}, \qquad
     \tilde{X}_i^L = \sum_{p=L+1}^\infty \langle \varepsilon_i, v_{p,\varepsilon} \rangle\, v_{p,\varepsilon}.
\end{align*}
By Parseval's identity, similarly as in Section~\ref{section_comparison}, in combination with Lemma~\ref{thm_gc_d_D} it holds as $n\to\infty$
\begin{equation}
 T_n^L :=  \max_{1 \leq k < n}\left\|\frac{1}{\sqrt{n}} \sum_{i=1}^n h\left(\frac{i-k}{n}\right)(X^L_i -\bar{X}^L_n) \right\|  \dto \sup_{0<t<1} \sqrt{\sum_{p=1}^L\lambda_{p,\varepsilon}G^2_p(t) }.  \label{eq_fd}
\end{equation}
The distribution on the right hand side of \eqref{eq_fd} is arbitrarily close  (in probability) for $L$ large enough to the limit distribution in \eqref{gc_functional_limit}.
Indeed, with Markov's inequality and the monotone convergence theorem and because $\{G_p\}$ are identically distributed, it holds for any $\zeta>0$ 
\begin{align*}
    \mathsf{P}\left(\sup_{0<t<1} \sqrt{\sum_{p=L+1}^\infty \lambda_{p,\varepsilon} G^2_p(t)} > \zeta \right) &\leq \frac{1}{\zeta^2} \sum_{p=L+1}^\infty \lambda_{p,\varepsilon}\, \mathsf{E}\left[ \sup_{0<t<1} G^2_p(t) \right] =\frac{1}{\zeta^2}\, \mathsf{E}\left[ \sup_{0<t<1} G^2_1(t)\right]\sum_{p=L+1}^\infty \lambda_{p,\varepsilon},
\end{align*}
which becomes arbitrarily small for $L$ large enough due to $\sum_{p\ge 1}\lambda_{p,\varepsilon}<\infty$.  Furthermore, the expectation can be bounded due to Assumption~\ref{ass_function_h1} as
\begin{align*}
    &\mathsf{E} \Big[\sup_{0<t<1} G^2_1(t)\Big] = \mathsf{E} \Big[\sup_{0<t<1} \Big(\int_0^{1-t}  B_1(1-t-y) \, dh(y)\Big)^2\Big] \leq\mathsf{E} \Big[ \operatorname{T\!V}^2(h) \sup_{0<x<1} |B_1(x)|^2\Big] \\
  &  = \operatorname{T\!V}^2(h)\,\mathsf{E} \Big[ \sup_{0<x<1} |B_1(x)|^2\Big]< \infty. 
\end{align*}
Next, we prove that $T_{n}^L$ is arbitrarily close to $T_{n,\operatorname{FF}}$ (in probability) uniformly in $n$ for $L$ large enough, i.e.,
 we want to show that for any $\tau,\eta > 0$, there exists $L_0(\tau,\eta)$ such that for all $n$
 \begin{equation}
     \mathsf{P}\left( \max_{1 \leq k < n} \left\|\frac{1}{\sqrt{n}} \sum_{i=1}^n h\left(\frac{i-k}{n}\right)(\tilde{X}^L_i -\bar{\tilde{X}}^L_n) \right\| > \tau   \right) \leq \eta  \label{eq_tail_pconv}
 \end{equation}
 for all $L\ge L_0(\tau,\eta)$, where $L_0(\tau,\eta)$ does not depend on $n$. Using again the Markov inequality
\begin{align}
   & \mathsf{P}\left( \max_{1 \leq k < n} \left\|\frac{1}{\sqrt{n}} \sum_{i=1}^n h\left(\frac{i-k}{n}\right)(\tilde{X}^L_i -\bar{\tilde{X}}^L_n) \right\| > \tau   \right) \leq \frac{1}{\tau^2} \,\E\max_{1 \leq k < n} \left\|\frac{1}{\sqrt{n}} \sum_{i=1}^n h\left(\frac{i-k}{n}\right)(\tilde{X}^L_i -\bar{\tilde{X}}^L_n) \right\|^2  \nonumber\\
    &\le\frac{1}{\tau^2} \sum_{p=L+1}^\infty \lambda_{p,\varepsilon}\, \E \max_{1 \leq k < n} \left(\frac{1}{\sqrt{n}} \sum_{i=1}^n h \left(\frac{i - k}{n}\right)\frac{\langle \varepsilon_i - \bar{\varepsilon}_n, v_{p,\varepsilon}\rangle}{\sqrt{\lambda_{p,\varepsilon}}}\right)^2.\label{eq_tau_odhad2}
\end{align}
We will study the expectations on the right hand side first. To this end, denote $\xi_{i,p} = \lambda_{p,\varepsilon}^{-\frac{1}{2}}\,\langle \varepsilon_i, v_{p,\varepsilon}\rangle$.  Then, setting $j=n-i+1$ and $l=n-k$, by a brief calculation using $h(x)=0$ for all $x\le 0$, 
\begin{align}
    &\max_{1 \leq k < n} \left(\frac{1}{\sqrt{n}} \sum_{i=1}^n h \left(\frac{i - k}{n}\right)(\xi_{i,p} - \bar{\xi}_{p,n})\right)^2
    = \max_{1 \leq l < n} \left(\frac{1}{\sqrt{n}} \sum_{j=1}^n h \left(\frac{l-j+1}{n}\right)(\xi_{n-j+1,p} - \bar{\xi}_{p,n})\right)^2\notag\\
    &\le 2 \max_{1 \leq l < n} \left(\frac{1}{\sqrt{n}} \sum_{j=1}^n h \left(\frac{l-j+1}{n}\right)\xi_{n-j+1,p} \right)^2
    + 2 \,\left(\bar{\xi}_{p,n}\right)^2\, \max_{1 \leq l < n} \left(\frac{1}{\sqrt{n}} \sum_{j=1}^n h \left(\frac{l-j+1}{n}\right)\right)^2
    \notag\\
    &\le 2\,\max_{1\le l<n}\left( \sum_{i=1}^l \frac{1}{\sqrt{n}}\sum_{j=1}^n\left(h\left(\frac{i-j+1}{n}\right)-h\left(\frac{i-j}{n}\right)\right)\, \xi_{n-j+1,p} 
    \right)^2 +2\,\sup_{0\le x\le 1}|h(x)|\, \left(\sqrt{n}\,\bar{\xi}_{p,n}\right)^2.\label{eq_th32_n1}
\end{align}
For the second summand, we can bound the expectation by
\begin{equation}\label{eq_thm_grad_secondterm}
E\left[2\,\sup_{0\le x\le 1}|h(x)|\, \left(\sqrt{n}\,\bar{\xi}_{p,n}\right)^2\right]\leq 2\,\sup_{0\le x\le 1}|h(x)|\frac{1}{n}\sum_{i=1}^n\sum_{i=1}^n\big|\cov(\xi_{i,p},\xi_{j,p})\big|\leq 2\operatorname{T\!V(h)}\sum_{h\in\Z}|\gamma_p(h)|.
\end{equation}
 with $\gamma_p(i)=\cov(\xi_{0,p},\xi_{i,p})$, because $\sup_{0\le x\le 1}|h(x)|\leq \operatorname{T\!V(h)}<\infty$.

It remains to upper bound the expectation of the first term in \eqref{eq_th32_n1}. To this end, define
\begin{equation*}
    Z_{i,p} = \frac{1}{\sqrt{n}} \sum_{j=1}^n \left(h\left(\frac{i-j+1}{n}\right) - h\left(\frac{i-j}{n}\right)\right) \xi_{n-j+1,p},\end{equation*}
and define the function $h_{\operatorname{T\!V}}:[-1,1]\rightarrow\R$ by $h_{\operatorname{T\!V}}(x)=\operatorname{T\!V}(h_{|[-1,x]})$ (the total variation of $h$ restricted to the interval $[-1,x]$). 
Clearly, for $-1\le x\le y\le 1$, it holds $h_{\operatorname{T\!V}}(x)+|h(y)-h(x)|\le h_{\operatorname{T\!V}}(y)$ and $h_{\operatorname{T\!V}}(1)=\operatorname{T\!V}(h)<\infty$. In particular, $h_{\operatorname{T\!V}}$ is indeed finite and non-decreasing.
For $a,b \in \N$, $a<b$, we have, with $\alpha$ and $\beta$ in Assumption~\ref{ass_function_h1},
\begin{align*}
    &\mathsf{E} \left(\sum_{i=a}^b Z_{i,p}\right)^2 = \mathsf{E} \left(\frac{1}{\sqrt{n}} \sum_{j=1}^n \left(h\left(\frac{b-j+1}{n}\right) - h\left(\frac{a-j}{n}\right)\right) \xi_{n-j+1,p}\right)^2 \\
    &= \frac{1}{n} \sum_{i,j=1}^n \left(h\left(\frac{b-i+1}{n}\right) - h\left(\frac{a-i}{n}\right)\right) \left(h\left(\frac{b-j+1}{n}\right) - h\left(\frac{a-j}{n}\right)\right) \mathsf{E} \xi_{n-i+1,p}\,\xi_{n-j+1,p} \\
    &\leq \beta \left(\frac{b-a+1}{n}\right)^{\alpha} \frac{1}{n}\sum_{j=1}^n \Big|h\Big(\frac{b-j+1}{n}\Big) - h\Big(\frac{a-j}{n}\Big)\Big|\sum_{i=1}^n |\gamma_p(j-i)|\\
    &\leq \beta \left(\frac{b-a+1}{n}\right)^{\alpha} \frac{1}{n}\sum_{j=1}^n \Big(h_{\operatorname{T\!V}}\Big(\frac{b-j+1}{n}\Big) - h_{\operatorname{T\!V}}\Big(\frac{a-j}{n}\Big)\Big)\sum_{h\in\Z}|\gamma_p(h)|\\
    &\leq \beta \operatorname{T\!V}(h) \left(\frac{b-a+1}{n}\right)^{1+\alpha} \, \sum_{h\in\Z}|\gamma_p(h)|,
\end{align*}
because, using that $h(x)=0$ and consequently $h_{\operatorname{T\!V}}(x)=0$ for all $x\le 0$,
\begin{align*}
&\sum_{j=1}^n \Big(h_{\operatorname{T\!V}}\left(\frac{b-j+1}{n}\right) - h_{\operatorname{T\!V}} \left(\frac{a-j}{n}\right)\Big)
    =\sum_{j=1}^b h_{\operatorname{T\!V}} \left(\frac{b-j+1}{n}\right)-\sum_{j=1}^{a-1}h_{\operatorname{T\!V}}\left(\frac{a-j}{n}\right)\\
 &   =\sum_{i=1}^bh_{\operatorname{T\!V}}\left(\frac i n \right)-\sum_{i=1}^{a-1}h_{\operatorname{T\!V}}\left(\frac i n\right)= \sum_{i=a}^{b} h_{\operatorname{T\!V}}\left(\frac in \right) \leq \operatorname{T\!V}(h) \,(b-a+1). 
\end{align*}
By \cite{moricz1976moment}, Theorem~1, there exists a constant $D$ depending only on the exponents (in our case 2 and 1+$\alpha$) such that
\begin{equation*}
\mathsf{E}\left(\max_{1 \leq l \leq n } \left|\sum_{i=1}^l Z_{i,p}\right|\right)^2 \leq D \cdot \beta \cdot \operatorname{T\!V}(h) \cdot \sum_{h\in\Z}|\gamma_p(h)|.
\end{equation*}
Together with \eqref{eq_th32_n1} and \eqref{eq_thm_grad_secondterm} this shows that the right hand side of \eqref{eq_tau_odhad2} is bounded 
by
\begin{align*}
&\frac{1}{\tau^2} \sum_{p=L+1}^\infty \lambda_{p,\varepsilon}\, \E \max_{1 \leq k < n} \left(\frac{1}{\sqrt{n}} \sum_{i=1}^n h \left(\frac{i - k}{n}\right)\frac{\langle \varepsilon_i - \bar{\varepsilon}_n, v_{p,\varepsilon}\rangle}{\sqrt{\lambda_{p,\varepsilon}}}\right)^2\\
\leq&\frac{1}{\tau^2} \sum_{p=L+1}^\infty \lambda_{p,\varepsilon}(D \cdot \beta+2) \cdot \operatorname{T\!V}(h) \cdot \sum_{h\in\Z}|\gamma_p(h)|=\frac{(D \cdot \beta+2) \cdot \operatorname{T\!V}(h)}{\tau^2} \sum_{p=L+1}^\infty \sum_{h\in\Z}\big|\cov(\langle \varepsilon_0, v_{p,\varepsilon} \rangle,\langle \varepsilon_h, v_{p,\varepsilon} \rangle)\big|.
\end{align*}
By our Assumption \ref{ass_covop}, we have $\sum_{p=1}^\infty \sum_{h\in\Z}|\cov(\langle \varepsilon_0, v_{p,\varepsilon} \rangle,\langle \varepsilon_h, v_{p,\varepsilon} \rangle)|<\infty$, so this bound can be made arbitrarily small by choosing $L$ large enough, completing the proof of \eqref{eq_tail_pconv}. Combining the arguments with Theorem 4.1 of \cite{billingsley1968} yields assertion (a).
The extension for (b) can be done in a similar manner as in the proof of Theorem~\ref{thm_asymptotics_h0}: First, analogously to (a),
\begin{equation*}
T_{n,\operatorname{WF}}(h,C_\varepsilon)\dto \sup_{0<t<1} \sqrt{\sum_{p=1}^\infty\frac{\lambda_{p,\varepsilon} }{\lambda_{1,\varepsilon} +\lambda_{p,\varepsilon} }G_p^2(t)}.
\end{equation*}
Now, we show that the difference of $T_{n,\operatorname{WF}}(h,C_\varepsilon)$ and $T_{n,\operatorname{WF}}(h,\hat{C})$ is negligible. Indeed,
\begin{align*}
&\left|T_{n,\operatorname{WF}}(h,C_\varepsilon)-T_{n,\operatorname{WF}}(h,\hat{C})\right|\leq  \max_{1\leq k< n} \left\|\left((\hat{C}+\hat{\lambda}_1 \operatorname{Id})^{- 1/2 }-(C_\varepsilon+\lambda_{1,\varepsilon} \operatorname{Id})^{- 1/2} \right)\frac{1}{\sqrt{n}}\sum_{i=1}^n(X_i-\bar{X}) \right\| \\
=&  \max_{1\leq k< n} \ \left\|(\hat{C}+\hat{\lambda}_1 \operatorname{Id})^{- 1/2  }\left((C_\varepsilon+\lambda_{1,\varepsilon} \operatorname{Id})^{ 1/2 }-(\hat{C}+\hat{\lambda}_1 \operatorname{Id})^{ 1/2 } \right)(C_\varepsilon+\lambda_{1,\varepsilon} \operatorname{Id})^{- 1/2 }\frac{1}{\sqrt{n}}\sum_{i=1}^n(X_i-\bar{X}) \right\| \\
\leq &\|(\hat{C}+\hat{\lambda}_1 \operatorname{Id})^{- 1/2  }\| \left\|(\hat{C}+\hat{\lambda}_1 \operatorname{Id})^{ 1/2 } -(C_\varepsilon+\lambda_{1,\varepsilon} \operatorname{Id})^{ 1/2 }\right\|  \max_{1\leq k< n} \Big\|(C_\varepsilon+\lambda_{1,\varepsilon} \operatorname{Id})^{- 1/2 }\frac{1}{\sqrt{n}}\sum_{i=1}^n(X_i-\bar{X})\Big\| \\
=& \hat{\lambda}_1^{-1/2} \left\|(\hat{C}+\hat{\lambda}_1 \operatorname{Id})^{ 1/2 } -(C_\varepsilon+\lambda_{1,\varepsilon} \operatorname{Id})^{1/2}\right\|  \cdot {T}_{n,\operatorname{WF}}(C_{\varepsilon})=   o_P(1),
\end{align*}
where we used \eqref{eq_value_inversesquarerooteigenvalue} and
\eqref{eq_consistency_cov_square_roots} from the proof of Theorem~\ref{thm_asymptotics_h0}. This completes the proof of (b).
\end{proof}

\begin{proof}[Proof of Theorem~\ref{th_gradual_PCA}]
By the Cauchy-Schwarz inequality it holds for any $p=1,\ldots,d$
\begin{align}
 &  \max_{1\le k<n}\left| \frac{1}{\sqrt{n}} \sum_{i=1}^n h\left(\frac{i-k}{n}\right) \langle \varepsilon_i - \bar{\varepsilon}_n, s_p\,\hat v_{p,\varepsilon} \rangle
   -
 \frac{1}{\sqrt{n}} \sum_{i=1}^n h\left(\frac{i-k}{n}\right) \langle \varepsilon_i - \bar{\varepsilon}_n, v_{p,\varepsilon} \rangle
   \right|\notag\\
   &= \max_{1\le k<n}\left|\left\langle  \frac{1}{\sqrt{n}} \sum_{i=1}^n h\left(\frac{i-k}{n}\right)\left( \varepsilon_i - \bar{\varepsilon}_n\right) \,,\, s_p\,\hat v_{p,\varepsilon} - v_{p,\varepsilon} \right\rangle
   \right|\notag\\
   &\le 
   \left\| s_p\,\hat v_{p,\varepsilon} - v_{p,\varepsilon}\right\|\cdot
   \max_{1\le k<n}\left\|\frac{1}{\sqrt{n}} \sum_{i=1}^n h\left(\frac{i-k}{n}\right)\left( \varepsilon_i - \bar{\varepsilon}_n\right)\right\|
   =o_P(1)\cdot O_P(1)=o_P(1)\label{eq:replace_DR}
\end{align}
by \eqref{eq_ass_est_grad} and Theorem~\ref{thm_gc_functional} (a), noting that the second term equals the fully functional statistics under the null hypothesis.
Consequently, the limit distribution of the dimension reduced statistic based on the estimated eigenfunctions $\hat v_{p,\varepsilon}$  equals the limit distribution of the dimension reduced statistics based on the true eigenfunctions $v_{p,\varepsilon}$.  The result then follows from an application of Proposition~\ref{thm_wi_d_D} and Lemma~\ref{thm_gc_d_D}, in combination with the consistency of the estimated eigenvalues guaranteed by \eqref{eq_ass_est_grad}. 
\end{proof}

 \begin{proof}[Proof of Proposition \ref{prop_alt_grad}]
     By assumption, there exists $0<H\le 1$ such that $h(x)> 0$ or $h(x)<0$ for all $0<x<H$, where we consider $h(x)>0$ w.l.o.g.
     Because $g$ is non-constant by assumption, neither is $f(x)=g(x)-\int_0^1g(x)\,dx$. Define $C=\sup(0< x< 1:f(x)\neq 0)$. Then, by the piecewise continuity of $g$, there exists $0\le c<C$ such that  $f(x)> 0$ for all $x\in(c,C)$ or $f(x)<0$ for all $x\in(c,C)$, where we consider the case $f(x)>0$ w.l.o.g. Then, for $t^\star=\max(c,C-H)$, it holds $h(x-t^{\star})f(x)>0$ for all $x\in (t^{\star},C)$ and $h(x-t^{\star})f(x)=0$ for $x\le t^\star$ as well as for $x>C$. Consequently,
\begin{align*}
    \max_{0\leq t\leq1}\left|\int_0^1 h(x-t)\, g(x) \, dx - \int_0^1g(x) \,dx  \int_0^1 h(x-t)\, dx \right| 
    \ge \left| \int_0^1h(x-t^\star)\, f(x)\,dx
    \right|>0.
\end{align*}
 \end{proof}

\begin{proof}[Proof of Theorem~\ref{thm_gradual_alternative}]
	By \eqref{eq:detect}, there exists $t^\star\in(0,1)$, such that 
\begin{equation*}
\left|\int_0^1 h(x-t^\star) g(x) \ dx - \int_0^1g(x) dx  \int_0^1 h(x-t^\star) dx \right|>0,
\end{equation*} 
and by the continuity of $h$, there exists some $c>0$ and $n_0\in\N$, such that for all $n\geq n_0$ and $k^\star=\lfloor nt^\star\rfloor$
\begin{equation*}
\left|\frac{1}{n}\sum_{i=1}^nh\Big(\frac{i-k^\star}{n}\Big)g\Big(\frac in\Big) -\frac{1}{n} \sum_{i=1}^nh\Big(\frac{i-k^\star}{n}\Big) \frac{1}{n}\sum_{j=1}^ng\Big(\frac jn\Big)\right|\geq c.
\end{equation*}
Consequently, for all $n\geq n_0$, we have that
    \begin{align}\label{eq_power_FF_1}
      &T_{n,\operatorname{FF}}(h) \notag\\
       \geq& \bigg\| \frac{1}{\sqrt{n}}
          \bigg[\Delta_n\bigg(\sum_{i=1}^nh\Big(\frac{i-k^\star}{n}\Big)g\Big(\frac in\Big) - \sum_{i=1}^nh\Big(\frac{i-k^\star}{n}\Big) \frac{1}{n}\sum_{j=1}^ng\Big(\frac jn\Big) \bigg)+\sum_{i=1}^nh\Big(\frac{i-k^\star}{n}\Big)\big(\varepsilon_i-\bar{\varepsilon}_n\big)\bigg]
          \bigg\|\notag\\
        \geq& \bigg\| \frac{1}{\sqrt{n}}
         \Delta_n\bigg(\sum_{i=1}^nh\Big(\frac{i-k^\star}{n}\Big)g\Big(\frac in\Big) - \sum_{i=1}^nh\Big(\frac{i-k^\star}{n}\Big) \frac{1}{n}\sum_{j=1}^ng\Big(\frac jn\Big) \bigg)  \bigg\|-\bigg\| \frac{1}{\sqrt{n}}\sum_{i=1}^nh\Big(\frac{i-k^\star}{n}\Big)\big(\varepsilon_i-\bar{\varepsilon}_n\big)
          \bigg\|\notag\\
\geq & c\sqrt{n} \left\|\Delta_n\right\|\, +O_P(1),
    \end{align}
    where the $O_P(1)$-term follows from Theorem \ref{thm_gc_functional}. This completes the proof for $T_{n,\operatorname{FF}}(h)$. 
   By  the submultiplicativity of the induced norm, it holds $\|A^{-1} x\| \geq \|A\|^{-1} \|x\|$. Consequently, 
    \begin{align*}
      &T_{n,\operatorname{WF}}(h,\hat{C}_n) 
       \geq \left\|(\hat{C}_n+\hat{\lambda}_{1,n}\operatorname{Id})^{1/2} \right\|^{-1} \,
       T_{n,\operatorname{FF}}(h)
       = \left(2\|\hat{C}\|\right)^{-1/2}\,
       \left(c\sqrt{n} \left\|\Delta_n\right\|\, +O_P(1)
       \right),
    \end{align*}
    completing the proof for $T_{n,\operatorname{WF}}(h,\hat{C}_n)$ by assumption.
\end{proof}

\begin{proof}[Proof of Theorem~\ref{thm_gradual_alternative_DR}]
As $\hat{\lambda}_p\le \|\hat C\|$ for all $p=1,\ldots,d$, we can lower bound the dimension reduced statistic by
\begin{align*}
  &  T^2_{n,\operatorname{PC}}(h;\hat{C})
      \geq \frac{n \|\Delta_n\|^2}{\|\hat{C}\|} \max_{1\le k< n } \sum_{p=1}^d  \bigg(\frac{1}{n\|\Delta_n\|}\sum_{i=1}^n h\Big(\frac{i-k}{n}\Big) 
     \left(\hat{\eta}_{i,p} - {\bar{\hat \eta}}_{n,p}\right)
      \bigg)^2,
\end{align*}
so it suffices to show that the maximum on the right hand side is stochastically bounded from below.
By Theorem~\ref{thm_gc_functional} (a) and the assumption that $n\|\Delta_n\|\rightarrow \infty$,  it holds
\begin{align*}
\frac{1}{\sqrt{n}\|\Delta_n\|}\,\max_{1\le k<n}\bigg\|\frac{1}{\sqrt{n}} \sum_{i=1}^n h\left(\frac{i-k}{n}\right)(
\varepsilon_i - \bar{\varepsilon}_n)\bigg\|=o_P(1),
\end{align*}
thus, by Parseval's identity, it holds
\begin{align}\label{eq_power_DC1}
&\max_{1\le k< n } \sqrt{\sum_{p=1}^d  \bigg(\frac{1}{n\|\Delta_n\|}\sum_{i=1}^n h\Big(\frac{i-k}{n}\Big) 
     \langle X_i - \bar{X}_n, \tilde v_{p} \rangle
    \bigg)^2}
\le
\max_{1\le k<n}\bigg\|\frac{1}{n\|\Delta_n\|} \sum_{i=1}^n h\left(\frac{i-k}{n}\right)\left[ X_i-\bar{X}_n\right]\bigg\|\\
&= \max_{1\le k<n}\bigg\|\frac{1}{n\|\Delta_n\|} \sum_{i=1}^n h\left(\frac{i-k}{n}\right) \Delta_n \Big(g\Bigl( \frac i n\Bigr)-\frac{1}{n}\sum_{j=1}^ng\Bigl( \frac j n\Bigr)\Big)\bigg\|+o_P(1)=O_P(1).\notag
\end{align}
Using the bound for the second term in \eqref{eq_power_DC1}, we get, analogously to 
 \eqref{eq:replace_DR}  by \eqref{eq_ass_est_grad_alt},
\begin{align*}
 &  \max_{1\le k<n}\left| \frac{1}{n\|\Delta_n\|} \sum_{i=1}^n h\left(\frac{i-k}{n}\right) \langle X_i - \bar{X}_n, \tilde s_p\,\hat v_{p} \rangle
   -
 \frac{1}{n\|\Delta_n\|} \sum_{i=1}^n h\left(\frac{i-k}{n}\right) \langle X_i - \bar{X}_n, \tilde v_{p} \rangle
   \right|\notag\\
   &\le 
   \left\| \tilde s_p\,\hat v_{p} -\tilde v_{p}\right\|\cdot
   \max_{1\le k<n}\left\|\frac{1}{n\|\Delta_n\|} \sum_{i=1}^n h\left(\frac{i-k}{n}\right)\left[ X_i-\bar{X}_n\right]\right\|
   =o_P(1).
 \end{align*}
  Consequently, by an application of the bound  for the first term in \eqref{eq_power_DC1} and the Cauchy-Schwarz-inequality,
\begin{align*}
&\max_{1\le k< n } \sum_{p=1}^d  \bigg(\frac{1}{n\|\Delta_n\|}\sum_{i=1}^n h\Big(\frac{i-k}{n}\Big) 
     \left(\hat{\eta}_{i,p} - {\bar{\hat \eta}}_{n,p}\right)   \bigg)^2
= \max_{1\le k<n}\sum_{p=1}^d\left(\frac{1}{n\|\Delta_n\|} \sum_{i=1}^n h\left(\frac{i-k}{n}\right) \langle X_i - \bar{X}_n, \tilde v_{p} \rangle\right)^2\, +o_P(1).
\end{align*}
We will now use similar arguments to complete the proof: As in Theorem~\ref{th_gradual_PCA} it holds
\begin{align}\label{eq_theorem32_for_alt_grad_proof}
    \max_{1\le k<n}\sum_{p=1}^d\left(\frac{1}{n\|\Delta_n\|} \sum_{i=1}^n h\left(\frac{i-k}{n}\right) \langle \varepsilon_i - \bar{\varepsilon}_n, \tilde v_{p} \rangle\right)^2=O_P\left(\frac{1}{n\|\Delta_n\|^2}\right)=o_P(1).
\end{align}
Furthermore,
\begin{align*}
 &\max_{1\le k< n }  \sum_{p=1}^d\bigg(\frac{1}{n\|\Delta_n\|}\sum_{i=1}^n h\left(\frac{i-k}{n}\right) \Big(g\Bigl( \frac i n\Bigr)-\frac{1}{n}\sum_{j=1}^ng\Bigl( \frac j n\Bigr)\Big)\langle \Delta_n, \tilde{v}_p \rangle\bigg)^2\\
 &=\frac{\sum_{p=1}^d\langle \Delta_n,\tilde{v}_p\rangle^2}{\|\Delta_n\|^2}\left(\sup_{1\le t\le 1}\left|\int_0^1h(x-t)\,g(x)\,dx-\int_0^1h(x-t)\,dx\cdot\int_0^1g(x)\,dx\right|+o(1)\right)^2,
\end{align*}
which is bounded from above by Parseval's identity as well as bounded away from zero by \eqref{eq_ass_grad_alt_2}. Thus, another application of the Cauchy-Schwarz-inequality  in combination with \eqref{eq_theorem32_for_alt_grad_proof} shows that
\begin{align*}
&\max_{1\le k<n}\sum_{p=1}^d\left(\frac{1}{n\|\Delta_n\|} \sum_{i=1}^n h\left(\frac{i-k}{n}\right) \langle X_i - \bar{X}_n, \tilde v_{p} \rangle\right)^2\\
&=\max_{1\le k< n }  \sum_{p=1}^d\bigg(\frac{1}{n\|\Delta_n\|}\sum_{i=1}^n h\left(\frac{i-k}{n}\right) \Big(g\Bigl( \frac i n\Bigr)-\frac{1}{n}\sum_{j=1}^ng\Bigl( \frac j n\Bigr)\Big)\langle \Delta_n, \tilde{v}_p \rangle\bigg)^2+o_P(1),
\end{align*}
which is bounded away from zero, thus completing the proof.

\end{proof}

\bibliography{bibliography}

\end{document}